\documentclass[11pt]{article}

\usepackage[colorlinks=true,breaklinks=true,bookmarks=true,urlcolor=blue,citecolor=blue,linkcolor=blue,bookmarksopen=false,draft=false]{hyperref}

\usepackage{amsmath}
\usepackage{graphicx}
\usepackage{psfrag}
\usepackage{amssymb}
\usepackage{bm}
\usepackage{mathtools}
\usepackage{verbatim,booktabs}
\usepackage{subfigure,epsfig,url}
\usepackage{float}
\usepackage{mathrsfs}
\usepackage{enumitem} 
\usepackage{caption}
\usepackage[usenames,dvipsnames]{pstricks}
\usepackage{epsfig}
\usepackage{pst-grad} 
\usepackage{pst-plot} 
\usepackage{xspace}
\usepackage{bbm}
\usepackage{comment}
\usepackage{tikz}
\usepackage{booktabs}
\usepackage{cleveref}
\usepackage{accents}
\usepackage{multirow}
\usepackage{array}
\usepackage{longtable}
\usepackage{pdflscape}
\usepackage[linesnumbered,ruled,vlined]{algorithm2e}
\usepackage[normalem]{ulem}

\usetikzlibrary{positioning,chains,fit,shapes,calc}
\definecolor{myblue}{RGB}{150,150,220}
\definecolor{mygreen}{RGB}{80,160,80}

\newcommand{\bR}{\mathbb R}

\newcommand{\bB}{\mathbb B}

\newcommand{\bO}{\mathcal O}

\newcommand{\fL}{\mathfrak M}

\newcommand{\cN}{\mathcal N}
\newcommand{\cM}{\mathcal M}
\newcommand{\bM}{\mathscr{M}}

\newcommand{\sos}{\mathsf{SoS}}

\newcommand{\maxcut}{\textsc{max-cut}\xspace}
\newcommand{\mincut}{\textsc{min-cut}\xspace}
\newcommand{\maxflow}{\textsc{max-flow}\xspace}

\newcommand{\maxstable}{\textsc{max-stable-set}\xspace}

\newcommand{\tsp}{\textsc{TSP}\xspace}

\DeclareMathOperator{\supp}{supp}
\DeclareMathOperator{\ssv}{sgnsupv}

\DeclareMathOperator{\nn}{\mathsf{NNS}}
\DeclareMathOperator{\ppn}{\mathsf{PS}}
\DeclareMathOperator{\snn}{\mathsf{nn}}
\DeclareMathOperator{\spn}{\mathsf{p}}

\DeclareMathOperator{\env}{env}

\DeclareMathOperator{\nnn}{\mathsf{NNS}^+}
\DeclareMathOperator{\nm}{\mathsf{NDS}^+}
\DeclareMathOperator{\ssc}{\mathsf{SSC}}

\DeclareMathOperator{\snm}{\mathsf{SoSC}}
\DeclareMathOperator{\minc}{\mathsf{MC}}
\DeclareMathOperator{\maxf}{\mathsf{MF}}
\newcommand{\deq}{\coloneqq}

\newcommand{\ie}{i.e., }
\newcommand{\eg}{e.g.,~}

\newcommand{\vc}[1]{\bm{#1}}
\newcommand{\ceil}[1]{\lceil#1\rceil}

\newcommand{\vx}{{\vc{x}}}
\newcommand{\vvu}{{\vc{u}}}
\newcommand{\va}{{\vc{\alpha}}}
\newcommand{\vy}{{\vc{y}}}
\newcommand{\vz}{{\vc{0}}}

\newcommand{\veu}{{\vc{1}}}

\providecommand{\vv}{}
\renewcommand{\vv}{{\vc{v}}}
\newcommand{\myvs}{{\vc{s}}}

\newcommand{\vth}{{\vc{\theta}}}

\newcommand{\vd}{{\vc{d}}}

\DeclareMathOperator{\mysl}{\vv_{\mathrm{s}}}
\DeclareMathOperator{\sr}{\vv_{\mathrm{t}}}

\usepackage{natbib}
\usepackage{amsthm}
\newtheorem{theorem}{Theorem}
\newtheorem{proposition}{Proposition}
\newtheorem{lemma}{Lemma}
\newtheorem{corollary}{Corollary}
\newtheorem{remark}{Remark}

\newtheorem{example}{Example}
\newtheorem{definition}{Definition}

\newcommand{\Halmos}{}

\DeclareMathOperator{\argmin}{argmin}

\newcommand{\orgdiv}[1]{#1}%
\newcommand{\orgname}[1]{#1}%
\newcommand{\orgaddress}[1]{#1}%
\newcommand{\postcode}[1]{#1}%
\newcommand{\city}[1]{#1}%
\newcommand{\country}[1]{#1}%

\newcommand{\fnm}[1]{#1}%
\newcommand{\sur}[1]{#1}%

\Crefname{chapter}{Chap.}{Chaps.}
\Crefname{section}{Sect.}{Sects.}
\Crefname{proposition}{Prop.}{Props.}
\Crefname{theorem}{Thm.}{Thms.}
\Crefname{definition}{Defn.}{Defns.}
\Crefname{corollary}{Cor.}{Cors.}
\Crefname{figure}{Fig.}{Figs.}
\Crefname{observation}{Obs}{Obss.}
\Crefname{Convention}{Conv.}{Convs.}
\Crefname{algocf}{Algorithm}{Algorithms}

\SetKwComment{Comment}{$\triangleright$\ }{}
\SetCommentSty{}
\SetKw{Continue}{continue}
\SetKw{Break}{break}

\allowdisplaybreaks

\title{Relaxations for binary polynomial
optimization via signed certificates}

\date{\today}

\author{\fnm{Liding} \sur{Xu} \thanks{\orgname{Zuse Institute Berlin}, \orgaddress{\city{Berlin}, \country{Germany}}.
             E-mail: {\tt lidingxu.ac@gmail.com}}
             \and
\fnm{Leo} \sur{Liberti} \thanks{\orgdiv{LIX CNRS}, \orgaddress{\city{Palaiseau}, \postcode{91128}, \country{France}}.
             E-mail: {\tt liberti@lix.polytechnique.fr}}
}

\begin{document}

\maketitle

\begin{abstract}
 We consider the problem of minimizing a polynomial $f$ over the (binary) hypercube. We show that, for a specific set of polynomials, their binary non-negativity (i.e., on the hypercube) can be checked in polynomial time via minimum cut algorithms, from which we construct a linear programming representation for this set. We categorize binary polynomials according to their signed support patterns and develop parameterized linear programming representations of binary non-negative polynomials. This allows the construction of signed binary non-negative certificates with adjustable signed support patterns and representation complexities, and we propose a method for minimizing $f$ by decomposing it as a sum of signed certificates.
 This method yields new hierarchies of linear programming relaxations for binary polynomial optimization. Moreover, since our decomposition depends only on the support of $f$, the new hierarchies are sparsity-preserving.
\end{abstract}


%


\section{Introduction.}

Given an $n$-variate polynomial $f$, the non-negativity problem over a set $S$ asks whether $f$ is non-negative over $S$. We let $S$ be the binary hypercube $\{0,1\}^n$. We call polynomials that are non-negative over this hypercube \emph{binary non-negative polynomials}, and we call the associated decision problem the \emph{binary non-negativity problem} (BNP).

The \textsc{binary polynomial optimization} (BPO) problem $\min\limits_{\vx \in \{0,1\}^n} f(\vx)$ has deep connections with combinatorial optimization (see \citet{chlamtac2012convex,de2012algebraic,lasserre2002polynomials}). The BNP and the BPO can be understood through their dual relationship. The BPO is equivalent to its dual conic optimization problem, which seeks the maximum $\lambda \in \mathbb{R}$ such that $f-\lambda$ is binary non-negative. Since binary non-negative polynomials form a convex cone, the BNP is precisely the feasibility problem of this dual formulation. We remark that linear optimization over the binary non-negative cone is generally intractable, because it would be equivalent to solving the BPO, which is known to contain many $\mathbf{NP}$-hard subclasses.

Given the aforementioned perspective, a common approach is to address the BPO approximately.  Given a conic inner approximation of the binary non-negative cone and an efficient algorithm to optimize over it, the relaxed optimal value over the inner approximation is a lower bound for the optimal value of the BPO. The elements of the inner approximation are called \textit{binary non-negativity certificates} because membership of $f$ (or $f-\lambda$) in the set of these certificates certifies feasibility for the BNP.

The Lasserre and Sherali-Adams hierarchies, for example, are well-known successive convex relaxations of the BPO based on binary non-negativity certificates (see \citet{lasserre2002explicit,sherali1990hierarchy}). The Lasserre hierarchy uses sum-of-squares (SoS) polynomials as binary non-negativity certificates, and relaxations can be solved via semidefinite programming (SDP) (see \citet{lasserre2002explicit,lasserre2002semidefinite,parrilo2003semidefinite,putinar1993positive}). The Sherali-Adams hierarchy uses sum-of-factor-products (SoFP) polynomials as binary non-negativity certificates, and its relaxations can be solved via linear programming (LP) (see \citet{krivine1964anneaux,lasserre2002semidefinite,sherali1992global,sherali1997new}).

The Lasserre hierarchy has been shown to produce tighter relaxations than the Sherali-Adams hierarchy (see \citet{laurent2003comparison}). We note that the size of the smallest Lasserre relaxation increases exponentially with the dimension $n$ and the degree $d$ of $f$: the size of the matrices in the SDP is $\bO\genfrac(){0pt}{2}{n}{d}$. Various  methods that can exploit sparsity (see \citet{waki2006sums,wang2020chordal,wang2021tssos}) and symmetry (see \citet{gatermann2004symmetry}) of SoS polynomials have been proposed to reduce the SDP size.   However, the complexity of these methods still depends on the  degree of  $f$.
 The size of the smallest   Sherali-Adams relaxation also increases exponentially with $n$ and $d$. However, higher-order Sherali-Adams relaxations can be solved in practice, since LP solvers are more efficient than SDP solvers. A tailored reduction method for the  Sherali-Adams hierarchy is based on constraint implication (see \citet{sherali2012reduced}). The sparsity exploitation method proposed by \citet{bienstock2018lp} leverages graph-theoretical tree decomposition, similar to the method for handling sparse SDPs in the Lasserre hierarchy.

To overcome the issue of high-degree polynomials, some approaches develop sparse non-negativity certificates for polynomial optimization. One example is the sum-of-non-negative-circuits (SoNC) polynomials; see  \citet{dressler2017positivstellensatz}.  A circuit polynomial is of the form  $\sum_{\va \in A}g_{\va} \vx^{\va} - g_{\va'} \vx^{\va'} $, where $A$ contains integral vertices of a simplex in $\bR^n$,  $\vx^\va := \prod_{j \in \{1,\dots, n\}}x_j^{\alpha_j}$  with  $g_\va > 0$ for all $\va \in A$,  and $\va'$  lies in the relative interior of this simplex. Circuit polynomials have a simple signed support pattern: at most one  monomial has a negative coefficient, while the other monomials have positive coefficients.  SoNC polynomials have  several convex programming  representations (see \citet{dressler2019approach,ghasemi2012lower,magron2023sonc}).

In this paper, we propose a new class of sparse binary non-negativity certificates for the BPO. We explore the discrete properties of binary polynomials and their signed support patterns, and construct BPO relaxations for arbitrary binary polynomials. Our study is motivated by several observations about polynomial optimization.

First, we observe that,  in the broader setting  of  semi-algebraic optimization (see \citet{lasserre2015introduction,parrilo2020sum}), the non-negativity of SoS, SoNC, and SoFP  polynomials can hold over larger domains than the binary hypercube, such as hypercubes, spheres, and semi-algebraic sets. This view unifies  various non-negative certificates under the framework of semi-algebraic optimization (see \citet{laurent2009sums}). The Sherali-Adams hierarchy is a special case of the Handelman hierarchy (see \citet{handelman1988representing,laurent2014handelman,laurent2023effective}) for general polynomial optimization, as the latter uses more redundant polynomial constraints to construct non-negativity certificates. When restricting to the BPO, the corresponding domain  shrinks to the binary hypercube, and the  Sherali-Adams hierarchy is the same as the Handelman hierarchy. Using the binary property that $x_j^2 = x_j$, one can  convert those non-negative polynomials to multilinear polynomials and use them as binary non-negativity certificates.  The present paper focuses on multilinear polynomials that are non-negative over the binary hypercube, rather than on complicated structures in constrained polynomial optimization.

Second, for a binary polynomial $f$, its signed support vector in $\{-1,0,1\}^{\bB}$ encodes the sign of the coefficient of $\vx^\va$ for each exponent vector $\va \in \{0,1\}^n$. We consider a subclass of binary polynomials with a specific signed support pattern, known as \emph{nonlinearly negatively signed (NNS) polynomials}\footnote{We warn the reader that some of the multi-word entity names in this paper may appear confusing at first. For example, we shall consider ``binary non-negative NNS polynomials''. The second `N' of ``NNS'' stands for ``non-negative'': an adjective that is already part of the name. But the first occurrence states that the polynomial evaluates to non-negative values over arguments consisting of binary vectors, and the second states that all of its monomials with degree $\ge 2$ have a coefficient $<0$. Although such a naming scheme requires some attention, it is standard. Changing it  would end up causing even more confusion.}. NNS polynomials are such that all nonlinear monomials have negative coefficients, while their linear parts can have arbitrary signs. NNS polynomials are submodular functions with nice algorithmic properties (see \citet{crama2011boolean}). We shall also refer to the following related classes: the \textit{negatively signed} (NS), \textit{positively signed} (PS), \textit{nonlinearly positively signed} (NPS), and \textit{nonlinearly differently signed} (NDS) polynomials, whose definitions follow the same sign-pattern convention.

\subsection{Our contribution.}
Our key innovation is a class of new relaxations for BPO based on the following ideas.
We show that the binary non-negativity of an NNS polynomial with $m$ monomials, including the constant and linear monomials, and degree $d$ can be checked in $\bO(m^2d)$ time via Orlin's \maxflow algorithm. This leads to an $\bO(md)$-sized  LP formulation of the cone of binary non-negative NNS polynomials.
For an arbitrary $n$-variate polynomial, we can decompose it as the sum of an NNS polynomial and a PS polynomial. We introduce two methods to construct piecewise linear concave extensions of the PS component.  These methods lead to an LP characterization of the binary non-negativity of an arbitrary polynomial $f$, as well as to the LP reformulation of the BPO for $f$. We further approximate the LP formulation of huge size by smaller relaxations generated from decompositions of $f$ into binary non-negativity certificates with simpler signed support patterns and lower representation complexity. Finally, we construct hierarchies of LP relaxations for the BPO. The standard signed hierarchy converges in at most $n+1$ levels, while the Lovász signed hierarchy converges in at most $\ceil{\log_2 n}+1$ levels. Each level of the relaxations can be solved in time doubly exponential in its level number $i$.

 \subsection{Outline of this paper.}
In the remaining part of this section, we review related work and introduce some notation.  In \Cref{sec.prelinaries}, we give some preliminary notions: submodularity, signed support decomposition, and the classification of binary polynomials.  In \Cref{sec.para}, we construct the polyhedral cone of binary non-negative NNS polynomials, and we study its separation algorithm and extended formulations. In \Cref{sec.concaveextensions}, we study the standard and the Lov\'asz concave extensions (or linearizations) of PS polynomials.  In \Cref{sec.lifting}, we combine the results from the previous two sections to study the BNP and BPO for arbitrary binary polynomials, and we characterize the  conic representation of binary non-negative polynomials. In \Cref{sec.refined}, we introduce the refined signed support decomposition. In \Cref{sec.nest}, we propose two hierarchies of relaxations for BPO based on the two concave extensions of PS polynomials mentioned above.
In \Cref{sec.cresult}, a computational test compares the Sherali-Adams, Lasserre, and our relaxations.  \Cref{sec.conc} concludes the paper.

\subsection{Related work.}
 The book by \citet{magron2023sparse} summarizes several proposals for sparse versions of the Lasserre hierarchy, including correlative sparsity (see \citet{grimm2007note,lasserre2006convergent,waki2006sums}) and term sparsity (see \citet{wang2019new,wang2021tssos}).

In the literature, we find some methods that look for a new parameterization of sparse non-negative polynomials, \eg SoNC polynomials (see \citet{dressler2017positivstellensatz,dressler2019approach}). The non-negativity of a circuit polynomial can be reduced to the non-negativity of a signomial through the exponential transformation (see \citet{murray2021newton}), and the latter problem can be solved by a convex relative entropy program proposed by \citet{chandrasekaran2016relative}.

 Given an $n$-variate degree-$d$ binary polynomial $f$ and a finite family $\{f^i\}$ of binary non-negativity certificates,  one can  usually use a convex program to search for a decomposition of $f$ into the sum of $f^i$.  In the Sherali-Adams hierarchy, since the constraints $0 \le x_j \le 1$ are valid for binary variables, the certificate $f^i$ is expressed as a bound-factor product form  $\prod_{j \in S_-}(1-x_j)\prod_{j  \in S_+}x_j$ (see \citet{handelman1988representing}), where $S_-,S_+$ are disjoint subsets of $\{1,\ldots,n\}$.  The upper bounds on the degree of $f_i$ and the size of the convex program are respectively $n$ and $\bO(2^{n})$ (see \citet{conforti2014integer,laurent2003comparison}). In the SoNC hierarchy, the certificate $f^i$ is a non-negative circuit polynomial. The upper bounds on the degree of $f_i$ and the size of the convex program are respectively  $n$ and $\bO(2^n)$ (see \citet{dressler2022optimization}). In the Lasserre hierarchy, the certificate polynomial $f^i$ is the square of a binary polynomial $g^i$. \citet{lasserre2002explicit} showed that the maximum degree of $g^i$ is bounded by $n$, and \citet{sakaue2017exact} gave a stronger bound by $\ceil{(n+d-1)/2}$. Therefore, the upper bound to the size of the convex program is either $\bO(2^{(n+d-1)})$ or $\bO(2^n)$.  We refer to \citet{laurent2003comparison} and \citet{lasserre2002semidefinite} for a comparison of the Sherali-Adams and Lasserre hierarchies.

Many exact algorithms (see \citet{helmberg1998solving,krislock2014improved,malick2013bridge,rendl2010solving}) for \maxcut problems and binary quadratic programs use convex relaxations and strong valid inequalities, whose properties are studied thoroughly in the book of \citet{deza1997geometry}. \citet{del2017polyhedral} introduced the notion of sparse support patterns for binary polynomials using hypergraphs, enabling a graph-theoretical approach to BPO. Furthermore, \citet{del2023complexity} and \citet{del2023polynomial} gave a polynomial-time approach to optimize beta-acyclic binary polynomials. These findings also suggest that the binary non-negativity of such polynomials can be verified in polynomial time.

There are many studies on the strength of BPO relaxation hierarchies. \citet{gouveia2010theta} studied  SDP-based relaxations for BPO through an SoS perspective.  For instance, the degree-2 SoS gives the relaxation of \citet{goemans1995improved} for \maxcut. \citet{lee2015lower} investigated lower bounds on the size of SDP relaxations for \maxcut, \maxstable, and
\tsp.  \citet{slot2023sum} analysed the error of the Lasserre hierarchy.  Other relaxation hierarchies in integer programming include  the Lovász-Schrijver hierarchy (see \citet{lovasz1991cones}), Balas-Ceria-Cornu{\'e}jols hierarchy  (see \citet{balas1993lift}), and Bienstock-Zuckerberg hierarchy (see \citet{bienstock2004subset}).

The submodularity of a pseudo-Boolean function  is characterized by  its second order derivatives (see \citet{billionnet1985maximizing,nemhauser1978analysis}).
Recognition of submodular/supermodular binary polynomials is  $\mathbf{co}$-$\mathbf{NP}$-complete (see \citet{crama1989recognition,gallo1989supermodular}). However, for submodular/supermodular binary polynomials of degree at most 3, their explicit form is exhibited by \citet{billionnet1985maximizing} and \citet{nemhauser1978analysis}, which makes the recognition problem easy. Another subfamily of submodular binary polynomials that can be recognized is the NNS ones, of interest in this paper. We refer to \citet{boros2002pseudo} and \citet{crama2011boolean} for topics on pseudo-Boolean functions, to \citet{fujishige2005submodular} for submodular function analysis, and to \citet{murota1998discrete} for discrete convex analysis.

\subsection{Notation and symbols}
We follow the standard convention for notation and symbols.
 We refer to \Cref{tab:notation} for descriptions of the symbols used in this paper. In this paper, most vectors are of two different dimensions: i) vectors in $\bB$ or $\bR^n$, \eg vector $\vx$ of variables; ii) vectors in $\{-1,0,1\}^\bB$ or $\bR^\bB$, \eg signed support vector or coefficient vector of a binary polynomial $f$.
\begin{table}[htbp]
  \centering
  \begin{tabular}{cp{0.85\linewidth}}
    \hline
    \textbf{Symbol} & \textbf{Description} \\
    \hline
    $\cN$ & the unordered set $[n] \deq \{1,\ldots,n\}$ \\
    $\bB$ & hypercube $\{0,1\}^n$ representing exponents of $n$-variate binary monomials \\
    $ \vz,\veu,\veu_j$ & all-zero, all-one, and $j$-th unit vectors in $\bB$ \\
    $|\vv|$ & $L_1$ norm of a vector $\vv \in \bB$ ($|\cdot|$ is also used as cardinality of a set) \\
    $z_j, z_J$ &  $j$-th entry of vector $\mathbf{z} \in \bR^{d}$, sub-vector of its entries indexed by $J \subseteq [d]$\\
    $\supp(\mathbf{z})$ & index set (support) of non-zero entries in $\mathbf{z}$\\
    $\bB_{\ell : u}$ & set $\{\va \in \bB:  |\va| \in [\ell, u]\}$ of vectors in $\bB$ with $L_1$ norm in $[\ell, u]$ \\
   $\bR(\vx)$ &  vector space of $n$-variate binary (multilinear) polynomials\\
   $f$ &  shorthand for  the binary polynomial $f(\vx)= \sum_{\va \in \bB} f_{\va} \vx^{\va} \in \bR(\vx)$ \\
   $\text{sgn}(f_\va)$ &  the sign (in $\{-1,0,1\}$) of  $f_\va$ \\
   $\myvs$ & the symbol usually used for signed support vectors in $\{-1,0,1\}^{\bB}$  \\
   $\ssv(f) $ & signed support vector $(\text{sgn}(f_\va) \;|\; \va\in \bB) \in \{-1,0,1\}^{\bB}$ of $f$ \\
   $\ssv(f)_S$ & sub-vector of the signed support vector of $f$ indexed by $S\subseteq\bB$   \\
   $\cN(\myvs)$  & set of var.~indices in all monomials supported by $\myvs \in  \{-1,0,1\}^{\bB}$ (Eq.~\eqref{eq.cns}) \\
   $x^+$  & shorthand for $\max(x,0)$ \\
   NS or PS & polynomial with all coefficients $\le 0$ or $\ge 0$ (\Cref{def.nondich}) \\
   NNS or NPS & same as above but only for coefficients of nonlinear monomials (\Cref{def.nondich}) \\
   NDS & neither an NNS nor an NPS polynomial (\Cref{def.nondich}) \\
   $\myvs' \preceq \myvs$ & partial order on signed support vectors (\Cref{def.porder}) \\
   $\ssc(\myvs)$ & set of polynomials with signed support vector $\preceq\myvs$ \\
   $\nnn(\myvs)$ & NNS
polynomials in  $\ssc(\myvs)$ that are non-negative on $\bB$ (\Cref{def.ns})\\
   $\Gamma(\myvs)$ & upper bound to concave extension complexity (\Cref{s:concext}) \\
   $\nm(\myvs)$ &  NDS polynomials  in  $\ssc(\myvs)$ that are non-negative on $\bB$ (\Cref{defn:nm}) \\
   $\sos^i(n)$ & $n$-variate sum-of-squares polynomials of degree $\le i$ \\
  $\snm(\Theta(\myvs))$ &   the sum of signed certificates (w.r.t. $\Theta(\myvs)$, \Cref{def.cn}) \\
    \hline
  \end{tabular}
    \caption{Meaning of some of the notation most often used in the paper.}
        \label{tab:notation}
\end{table}

\section{Preliminaries.}
\label{sec.prelinaries}
We present necessary definitions and known properties of binary polynomials.
In this paper, we assume that the vector space $\bR(\vx)$ is spanned by binary monomials (\ie monomials in binary variables). As $\bR(\vx)$ is isomorphic to $\bR^{\bB}$, we do not distinguish between the binary polynomials in $\bR(\vx)$ and the vectors in $\bR^{\bB}$. We first present a classification of binary polynomials based on their signed support vectors.

\begin{definition}
\label{def.nondich}
Let  $f \in \bR(\vx)$  be a  binary polynomial. Then,
\begin{itemize}
    \item  $f$ is an affine polynomial (a linear polynomial plus a constant) if $\ssv(f)_{\bB_{2:n}}  = \vz$;
    \item $f$ is  an NS (resp. PS) polynomial if $\ssv(f)  \le \vz$ (resp. $\ssv(f)  \ge \vz$);
    \item  $f$ is an NNS (resp. NPS) polynomial if $\ssv(f)_{\bB_{2:n}}  \le \vz$ (resp. $\ssv(f)_{\bB_{2:n}}  \ge \vz$).
    \item  $f$  is an NDS polynomial, if it is neither an NNS nor an NPS polynomial.
\end{itemize}
\end{definition}
We also extend these definitions to polynomials over $\bR^n$. For example, a circuit polynomial is a sum of a PS polynomial and a negative monomial, thus it is an NDS polynomial.

We next consider special classes of functions defined over $\bB$.  We say that a function $g: \bB  \to \bR$ is a \emph{submodular} (resp. \emph{supermodular}) function if
for every $\vx, \vy\in \bB$, $g(\vx) + g(\vy) \ge g\left ( \vx \lor \vy \right) + g\left(\vx \land \vy\right)$ (resp.  $g(\vx) + g(\vy) \le g\left ( \vx \lor \vy \right) + g\left(\vx \land \vy\right)$), where $\lor$ and $\land$ denote componentwise maximum and minimum, respectively. A \emph{modular function} is both submodular and supermodular. The classification in \Cref{def.nondich} allows for characterizing submodularity/supermodularity of binary polynomials as follows.

 \begin{lemma}\textup{\cite[Thm.~13.21]{crama2011boolean}}
 \label{prop.subchar}
 Every affine polynomial is a modular function. Every NNS (resp.~NPS) binary polynomial is a submodular (resp.~supermodular) function.
\end{lemma}

\citet{nemhauser1978analysis} showed that all submodular binary polynomials are NNS polynomials, when the degree is 2; \citet{billionnet1985maximizing} showed that,  when the degree is 3, some submodular binary polynomials are not NNS polynomials.

 We now introduce the \emph{signed support decomposition} of any binary (particularly NDS) polynomial as the sum of an NNS polynomial and a PS polynomial (see \citet{xu2023}), which is based on the following two transforms in $\bR(\vx)$:
\begin{align*}
    \nn: & \bR(\vx) \to \bR(\vx) \mbox{ mapping } f(\vx) \mapsto \nn(f)(\vx)  \mbox{ defined by } f_{\vz} + \sum_{\va \in \bB_{1:1}} f_\va \vx^{\va} + \sum_{\va \in \bB_{2:n}} \min(f_\va, 0) \vx^{\va}\\
   \ppn:&    \bR(\vx) \to \bR(\vx) \mbox{ mapping } f(\vx) \mapsto  \ppn(f)(\vx) \mbox{ defined by } \sum_{\va \in \bB_{2:n}} \max(f_\va, 0) \vx^{\va},
\end{align*}
where $\nn(f)$ is the NNS component of $f$ and $\ppn(f)$ is the PS component of $f$. Note that $f = \nn(f) + \ppn(f)$, and this  decomposition is unique.

We define a partial order on $\{-1,0,1\}^{\bB}$ in order to exploit the signed support patterns of binary polynomials.
\begin{definition}
\label{def.porder}
  Given $\myvs', \myvs \in \{-1,0,1\}^{\bB}$ we have  $\myvs' \preceq \myvs$ if: (i) for every $\va \in \bB_{0:1}$ we have $|s'_\va| \le |s_\va|$; (ii) for every $\va \in \bB_{2:n}$ we have either $0 \le s'_\va  \le s_\va $ or $0 \ge s'_\va  \ge s_\va $.
\end{definition}

Given $f \in \bR(\vx)$ and $\myvs \in \{-1,0,1\}^{\bB}$, we say that $f$ is signed-supported within $\myvs$ if the \emph{signed support constraint} $\ssv(f) \preceq \myvs$ w.r.t. $\myvs$ holds. The constraint can be expressed entry-wise as:
\begin{equation}
\label{eq.partialrep}
 \ssv(f) \preceq \myvs   \iff  \begin{cases}
  f_\va = 0 \textup{ (zero)} & \textup{for } \va \in \bB_{0:1}\; \mbox{ s.t. } s_\va = 0\\
    f_\va    \in \bR \textup{ (free)}& \textup{for } \va \in \bB_{0:1}\; \mbox{ s.t. } |s_\va| = 1\\
  f_\va = 0  \textup{ (zero)}& \textup{for } \va \in \bB_{2:n}\; \mbox{ s.t. } s_\va = 0 \\
    f_\va  \in \bR_+ \textup{ (non-negative)}& \textup{for } \va \in \bB_{2:n}\; \mbox{ s.t. } s_\va = 1\\
  f_\va  \in \bR_-  \textup{ (non-positive)}& \textup{for } \va \in \bB_{2:n}\; \mbox{ s.t. } s_\va = -1.
  \end{cases}
\end{equation}
This relation implies that the support of $f$ is a subset of the support of $\myvs$, and the signs of nonlinear monomials in $f$ cannot be opposite to those of $\myvs$.
 The constraint defines a convex cone $\ssc(\myvs)\deq \{f \in \bR(\vx): \ssv(f)  \preceq \myvs\}$ of sparse vectors (with $|\myvs|$ non-zero entries). We illustrate the aforementioned definitions with the following example.

\begin{example}
\label{exampl.pre}
Consider $f(\vx) =- x_2 x_3 -2 x_1 x_3 x_4 -5 x_3 x_5 +x_1 x_2  + 2 x_2 x_3 x_4  + 5x_4 x_5 + x_2 + x_3-x_4 + 7$. Then, $\nn(f)(\vx)  = - x_2 x_3 -2 x_1 x_3 x_4 -5 x_3 x_5 + x_2 + x_3-x_4 + 7$ and $\ppn(f)(\vx) = x_1 x_2  + 2 x_2 x_3 x_4  + 5x_4 x_5$. The nonzero entries of the signed support vector $\ssv(f)\in\{-1,0,1\}^{\{0,1\}^5}$ are
\begin{align*}
  \ssv(f)_{(0,1,1,0,0)}&=-1   &\ssv(f)_{(1,0,1,1,0)}&=-1   &\ssv(f)_{(0,0,1,0,1)}&=-1  \\
   \ssv(f)_{(1,1,0,0,0)}&=1 &\ssv(f)_{(0,1,1,1,0)}&=1  &\ssv(f)_{(0,0,0,1,1)}&=1   \\
  \ssv(f)_{(0,1,0,0,0)}&=1     &\ssv(f)_{(0,0,1,0,0)}&=1 &\ssv(f)_{(0,0,0,1,0)}&=-1 \\
    \ssv(f)_{(0,0,0,0,0)}&=1.     && &&
\end{align*}

\end{example}

\section{Binary non-negativity of NNS polynomials.}
\label{sec.para}

In this section, we consider binary non-negative NNS polynomials and study their representations. Previous work by \citet{billionnet1985maximizing,hansen1974lin,picard1982network} shows that one can minimize any binary NNS polynomial in time polynomial in its encoding size. We extend these results by providing a polynomially sized extended formulation of the set of non-negative NNS polynomials having a given signed support pattern.

Throughout this section, we fix a signed support vector $\myvs$ and consider the set of NNS polynomials signed-supported within $\myvs$. This allows for characterizing binary non-negative NNS polynomials through their signed support patterns. Next, we give the mathematical description of those NNS polynomials.

\begin{definition}
\label{def.ns}
Given $\myvs \in \{-1,0,1\}^{\bB}$ such that $s_{\bB_{2:n}} \le 0$, the set of binary non-negative NNS polynomials signed-supported within $\myvs$ is
\begin{equation}
\label{eq.expandcone}
\nnn(\myvs) \deq \{f \in \bR(\vx): f \in \ssc(\myvs) \land\ \forall  \vx \in \bB \; f(\vx) \ge 0\}.
\end{equation}
\end{definition}

We next study the properties of $\nnn(\myvs)$.
The set $\nnn(\myvs)$ is defined by two types of constraints. The first signed support constraint, $\ssv(f) \preceq \myvs$, specifies the signed support pattern of $f$, as in \eqref{eq.partialrep}. The second type of constraint, $\forall  \vx \in \bB, f(\vx) \ge 0$, stipulates the binary non-negativity of $f$. This constraint is called the \emph{non-negativity constraint}. Note that the binary vector $\vx$ in the non-negativity constraint serves  both as an index and as data for its \emph{linear sub-constraints}: $f(\vx) \ge 0$, where $f(\vx)=\sum\limits_{\va \in \bB} \vx^{\va} f_{\va}$ is interpreted as a linear function of the vector $f$.

Let $m\deq|\myvs|$ be the number  of non-zeros in $\myvs$, and let $d \deq \max\limits_{\va \in \supp(\myvs)}|\va|$.  Polynomials in $\nnn(\myvs)$ have at most $m$ monomials and have degree at most $d$. The set $\nnn(\myvs)$ is defined by a system of linear constraints over  \emph{monomial variables} $\{f_{\va}\}_{\va \in \supp(\myvs)}$, and it satisfies the following properties.

\begin{theorem}
 \label{thm.subnn}
For any $\myvs\in\{-1,0,1\}^{\bB}$ such that $\myvs_{\bB_{2:n}}\le \vz$,  $\nnn(\myvs)$ is a convex polyhedral cone with an LP formulation  of $m$ monomial  variables subject to a signed support constraint and a non-negativity constraint. Moreover, every NNS polynomial $f$ signed-supported within $\myvs$ is binary non-negative if and only if $f \in \nnn(\myvs)$.
 \end{theorem}
   \begin{proof}{Proof.}
    The constraint $\ssv(f) \preceq \myvs$ implies that $f$ has at most $m$ nonzero coefficients, while all coefficients outside $\supp(\myvs)$ are fixed to zero. Therefore, $\nnn(\myvs)$ has an LP formulation of $m$ monomial variables, so it is a convex polyhedral set. We know that $\nnn(\myvs)$ is the set of all binary non-negative NNS polynomials signed-supported within $\myvs$. For every $f \in \nnn(\myvs)$ and every positive scalar $\lambda$, $\lambda f$ is binary non-negative.
 \Halmos \end{proof}

\subsection{Separation algorithms for the non-negativity constraint.}
\label{subsec.sepalgo}
The next question is algorithmic: how can we optimize over the cone  $\nnn(\myvs)$? As the number of nontrivial linear sub-constraints equals $|\bB| = 2^n$, we can consider the \emph{separation problem} for $\nnn(\myvs)$, which asks whether all linear sub-constraints are satisfied for any given $f$ signed-supported within $\myvs$.  If not, it outputs the most violated constraint.  Thus, the  BNP for $f$  is equivalent to the separation problem between $\nnn(\myvs)$ and $f$.

We review several separation algorithms and illustrate a reduction approach based on \mincut. In particular, our presentation explicitly maintains the maps between the original and transformed problems, which facilitates the derivation of a novel LP formulation of $\nnn(\myvs)$.
The separation problem can be reduced to finding $\argmin_{\vx \in \bB}f(\vx)$. Throughout this paper, we assume that the coefficients of $f$ are fixed-size rationals.
As indicated by \Cref{prop.subchar}, the NNS  polynomial $f$ is submodular, so the optimization problem is a submodular minimization problem.  Therefore, the separation problem can be addressed by general-purpose submodular minimization algorithms. For example, employing  the algorithm of \citet{orlin2009faster}, one can minimize $f$ in $\bO(n^5md + n^6)$ time, where $\bO(md)$ is the time complexity of evaluating $f$.

 Another tailored method  further reduces the  minimization problem to a \mincut/\maxflow problem and uses the dedicated algorithms (see \citet{billionnet1985maximizing,hansen1974lin,picard1982network}), which have better  time complexities. Our reduction is based on the procedure proposed by \citet{picard1982network}. Let us write
\begin{equation}
\label{eq.sep.f}
    f(\vx) = f_{\vz} +  \sum_{\va \in A} f_{\va} \vx^{\va} +  \sum_{j \in \cN} f_{\veu_j} x_j,
\end{equation}
 where for all $\va \in A$, $|\va| \ge 2$ (\ie $\vx^{\va}$ is a nonlinear monomial). As $f$ is an NNS polynomial, for all $\va \in A$,  $f_{\va}$ is non-positive, and for all $j \in \cN$,  $f_{\veu_j}$ (for brevity, denoted as $f_j$) has an arbitrary sign.

 Then, we present several preprocessing steps for reducing the problem.
 Let
 \begin{equation}
 \label{eq.sep.fa}
     f^a \deq f_{\vz} + \sum_{\va \in A} f_{\va},
 \end{equation} and  let
  \begin{equation}
 \label{eq.sep.fb}
    f^b(\vx) \deq  \sum_{\va \in A} -f_{\va} (1-\vx^{\va}) +  \sum_{j \in \cN} f_j x_j = \sum_{\va \in A} f_{\va} \vx^{\va} +  \sum_{j \in \cN} f_j x_j - \sum_{\va \in A} f_{\va} = f(\vx) - (\sum_{\va \in A} f_{\va} + f_{\vz}).
 \end{equation}
Note that $f(\vx)= f^a+f^b(\vx)$ and  $f^a$ is independent of $\vx$. Therefore, we can reduce the minimization  of  $f$ to the minimization of $f^b$ over $\bB$.

Let $\cN_f \deq \{j \in \cN: f_j  \le 0\}$ denote the indices that will be involved in the next reduction step  for $f$. We will demonstrate that variables $x_j$ indexed by $\cN_f$ in any optimal solution can be set to one without loss of optimality. Additionally, we say that a binary vector $\vx$ is \emph{partly fixed}, if  its entries indexed by $\cN_f$ are ones. Moreover, we call these entries \emph{reducible}. Let $\bB_f$ denote the collection of partly fixed binary vectors within the binary hypercube $\bB$.
Let
 \begin{equation}
     \label{eq.sep.fc}
    f^{c}(\vx) \deq \sum_{\va \in A} -f_{\va} (1-\vx^{\va}) +  \sum_{j \in \cN} f^+_j x_j.
 \end{equation}
It follows that, for all  $\va \in A$, $-f_{\va} \ge 0$, for all $j \in \cN$, $f^+_j \ge 0$, and, in particular, for all $j \in \cN_f$, $f^+_j = 0$. Thus, we find that
\begin{equation}
\label{eq.fbfc}
\begin{aligned}
      f^b(\vx) &= \sum_{\va \in A} -f_{\va} (1-\vx^{\va}) +  \sum_{j \in \cN_f} f_j x_j + \sum_{j \in \cN \smallsetminus \cN_f} f_j x_j \\
      &=  \sum_{\va \in A} -f_{\va} (1-\vx^{\va}) +  \sum_{j \in \cN_f}  (f_j + f^+_j) x_j + \sum_{j \in \cN \smallsetminus \cN_f} f^+_j x_j \\
      &= f^c(\vx) +  \sum_{j \in \cN_f}  f_j x_j.
\end{aligned}
\end{equation}

We have the following observation on the optimal solutions to the minimization problems of the aforementioned reduced polynomials.

 \begin{lemma}
  \label{lem.threef}
$\argmin\limits_{\vx \in \bB_f}f^{c}(\vx) = \argmin\limits_{\vx \in \bB_f}f^{b}(\vx) = \argmin\limits_{\vx \in \bB_f}f(\vx) \subseteq \argmin\limits_{\vx \in \bB}f(\vx)$.
 \end{lemma}
 \begin{proof}{Proof.}
     Looking at $f$, entries $x_j$ indexed by $\cN_f$ have non-positive coefficients  $f_j$, and coefficients of nonlinear monomials of $f$ are all non-positive as well. Therefore, we can set those reducible entries $x_j$ in any optimal solution for $\min\limits_{\vx \in \bB}f(\vx)$  to one, \ie $\argmin\limits_{\vx \in \bB_f}f(\vx) \subseteq \argmin\limits_{\vx \in \bB}f(\vx)$. As $f^b$ is identical to $f$ (up to a constant), it follows that $\argmin\limits_{\vx \in \bB_f}f(\vx)=\argmin\limits_{\vx \in \bB_f}f^{b}(\vx)$. When restricting the search space to $\bB_f$, all variables $x_j$ indexed by $\cN_f$ are one. Therefore, $f^b(\vx) - f^c(\vx)$ is a constant for all $\vx \in \bB_f$. This implies that $\argmin\limits_{\vx \in \bB_f}f^{c}(\vx) = \argmin\limits_{\vx \in \bB_f}f^{b}(\vx)$ as claimed.
  \Halmos \end{proof}

The above lemma establishes the existence of  partly fixed binary vectors that minimize the functions $f$, $f^b$, and $f^c$ simultaneously.
 We can therefore minimize $f^c$, the linear terms of which have non-negative coefficients. This property is crucial, as \mincut problems usually require non-negative flows and capacities.  Subsequently, throughout our analysis, we exclusively focus on the search space of partly fixed binary vectors. This completes our first preprocessing step.

 We proceed by constructing a capacitated directed flow network  $G^c$, representing the binary polynomial $f^c$. In this network, the node set is  $V \deq A \cup \cN \cup \{\mysl, \sr\}$, where the additional nodes $\mysl$ and $\sr$ correspond to the source and target nodes, respectively. Edges go from $\mysl$ to $A$, from $A$ to $\cN$, and from $\cN$ to $\sr$. For each node $\va \in A$, there exists an edge $(\mysl, \va)$ with a positive capacity  $-f_{\va}$; for every node $\va \in A$ and  $j \in \supp(\va)$, an edge $(\va, j)$ is present with an unlimited capacity (numerically, a sufficiently large number denoted $\bM$);  for each node $j \in \cN$, there is an edge $(j, \sr)$ associated with a non-negative capacity $f^+_j$. In summary, the edge set of $G^c$ is  $E = \{(\mysl, \va)\}_{\va \in A} \cup \{(\va, j)\}_{\va \in A, j \in \supp(\va)} \cup \{(j,\sr)\}_{j \in \cN}$.

We next consider cuts in $G^c$. Any subset $S$ of $V$ defines a cut in $G^c$, whose capacity $C^c(S)$ is the sum of the capacities of edges from $S$ to $\bar{S}\deq V \smallsetminus S$.  We use a binary vector $\vvu \in  \{0,1\}^V$ to label the nodes of $S$: for $\vv \in V$, $\vvu_{\vv} = 1$ if and only if $\vv \in S$. For brevity, we use $\vvu$ to denote cuts, and denote $C^c(S)$ by $C^c(\vvu)$. In the following, we consider the cuts for which $\mysl \in S$ and $\sr \notin S$, \ie we focus on the set $U \deq \{\vvu \in  \{0,1\}^V: u_{\mysl} =1 ,u_{\sr} = 0\}$ of labels. The flow network is shown in \Cref{fig.flow}. We have the following theorem that connects the minimization of $f^c$ over partly fixed binary vectors to the \mincut problem in $G^c$.

\begin{theorem}
    \label{thm.mincut}
    $\min\limits_{\vx \in \bB_f}f^{c}(\vx) = \min\limits_{\vvu \in U}C^c(\vvu)$. In addition, since $m$ counts monomials including the constant and linear monomials, the \mincut problem in $G^c$ can be solved in $\bO(m^2d)$ time using Orlin's \maxflow algorithm.
\end{theorem}
The proof for \Cref{thm.mincut} and details on the separation algorithm are provided in the appendix.

\begin{figure}
    \centering
\begin{tikzpicture}[thick,
  fsnode/.style={fill=myblue,draw,circle},
  ssnode/.style={fill=mygreen,,draw,circle},
  every fit/.style={draw,inner sep=2pt,text width=1.5cm},
  ->,shorten >= 2pt,shorten <= 2pt
]

\begin{scope}[start chain=going below,node distance=4mm]
\node[fsnode,on chain] (f1) [label=left:] {};
\node[on chain] (f2) [label=left: ] {$\vdots$};
\node[fsnode,on chain] (f3) [label=above: {\footnotesize $u_\va = 0, \va \in \bar{S}$}] {\footnotesize $\va$};
\node[on chain] (f4) [label=left: ] {$\vdots$};
\node[fsnode,on chain] (f5) [label=left: ] {};
\end{scope}

\begin{scope}[xshift=6cm, start chain=going below,node distance=5mm]
\node[ssnode,on chain] (s1) [] {};
\node[on chain] (s2) []  {\footnotesize$\vdots$};
\node[ssnode,on chain] (s3) [label=above: {\footnotesize  $ u_{j_1} = 0, j_1 \in \bar{S}$}] {\footnotesize $j_1$};
\node[ssnode,on chain] (s4) [label=above: {\footnotesize  $u_{j_2} = 0, j_2 \in \bar{S}$}] {\footnotesize $j_2$};
\node[ssnode,on chain] (s5) [label=above: {\footnotesize  $ u_{j_3} = 0, j_3 \in \bar{S}$}] {\footnotesize $j_3$};
\node[on chain] (s6) []  {\footnotesize $\vdots$};
\node[ssnode,on chain] (s7) []  {};
\end{scope}

\node [fsnode, left=2cm of f3, label=above: {\footnotesize $u_{\mysl} = 1, S$}] (vs) {$\mysl$};
\node [ssnode, right=5cm of s4, label=above: {\footnotesize $u_{\sr} =0, \bar{S}$}] (vt) {$\sr$};
\node [myblue,fit=(f1) (f5),label=above:$A$] {};
\node [mygreen,fit=(s1) (s7),label=above: $\cN$] {};

\draw (f3) -- node[above]{$\bM$}(s3);
\draw (f3) -- node[above]{$\bM$}(s4);
\draw (f3) -- node[above]{$\bM$}(s5);
\draw (vs)  -- node[above]{\footnotesize $-f_{\va} \ge 0$} (f3);
\draw (s3) -- node[above]{$f^+_{j_1} \ge 0$}(vt);
\draw (s4) -- node[above]{$f^+_{j_2} \ge 0$} (vt);
\draw (s5) -- node[above]{$f^+_{j_3} \ge 0$} (vt);
\draw (vs) -- (f1);
\draw (vs) -- (f5);
\draw (s1) -- (vt);
\draw (s7) -- (vt);
\end{tikzpicture}
\caption{The network $G^c$ derived from $f^c$. Representative nodes are labeled by their membership in $S$ or $\bar{S}$ and by the corresponding value of $u$. The number $\bM$ denotes an unlimited capacity, equivalently enforced by excluding cuts with $u_\va=1$ and $u_j=0$ for some $j\in\supp(\va)$.}
\label{fig.flow}
\end{figure}

\subsection{An extended formulation of the non-negativity constraint.}
The
non-negativity constraint $\min\limits_{\vx \in \bB}f(\vx) \ge 0$  consists of $2^n$ linear sub-constraints. Based on the \mincut separation algorithm, we next present a novel extended formulation of this constraint of size $\bO(md)$. Our reformulation  is based on the LP formulation of \mincut and its duality.

  Recall that the \mincut problem in $G^c$ is equivalent to the reduced problem $\min\limits_{\vx \in \bB_f}f^c(\vx)$, but the non-negativity constraint places restrictions on $f$.  Therefore, we begin with the LP formulations of the reduced problem, and modify its LP formulations to represent the original problem $\min\limits_{\vx \in \bB}f(\vx)$.

First, the standard LP formulation of the \mincut problem $\min\limits_{\vvu \in U}C^c(\vvu)$ is as follows:
\begin{subequations}
\label{flow.lp01primal}
\begin{align}
\minc(G^c) ~\deq ~&\min \quad  & \sum_{\va \in A} (-f_{\va} d_{\mysl\va}) + \sum_{\va \in A \atop j \in \supp(\va)}\bM d_{\va j}  \quad&+\quad \sum_{j \in \cN} f^+_j d_{j \sr}  \label{flow.lp01primal.obj}  \\
\text{s.t.}\; & \forall  \va \in A,  & d_{\mysl\va} -  1 + u_{\va}  \quad &\ge \quad 0  \label{flow.lp01primal.c1}   \\
&\forall \va \in A,   j \in  \supp(\va), & d_{\va j} - u_{\va} + u_j \quad &\ge\quad 0   \label{flow.lp01primal.c2} \\
& \forall  j \in \cN, &  d_{j\sr} - u_j  \quad & \ge  \quad 0  \label{flow.lp01primal.c3} \\
& \forall  \vv \in A \cup \cN, & u_{\vv} \quad & \in  \quad \bR  \label{flow.lp01primal.c5} \\
&\forall  (\vv,\vv') \in E, &   d_{\vv\vv'} \quad & \ge \quad 0,  \label{flow.lp01primal.c4}
\end{align}
\end{subequations}
where $\vvu$ labels the  cuts.

We make a simple modification to the \mincut problem by changing the capacities $f^+_j$ of edges $(j,\sr)$  to $f_j$. This creates a new flow network $G^b$. The  following formulation is adapted from \eqref{flow.lp01primal}:
\begin{subequations}
\label{flow.lp2primal}
\begin{align}
\minc(G^b) ~\deq ~&\min \quad  & \sum_{\va \in A} (-f_{\va} d_{\mysl\va})   + \sum_{j \in \cN} f_j d_{j \sr} &  \label{flow.lp2primal.obj}  \\
\text{s.t.}\; & \forall  \va \in A,  & d_{\mysl\va} -  1 + u_{\va}  \quad &\ge \quad 0  \label{flow.lp2primal.c1}   \\
&\forall \va \in A,   j \in  \supp(\va), &  - u_{\va} + u_j \quad &\ge\quad 0   \label{flow.lp2primal.c2} \\
& \forall  j \in \cN, &  d_{j\sr} - u_j  \quad &\ge \quad 0  \label{flow.lp2primal.c3} \\
& \forall  \vv \in A \cup \cN, & u_{\vv} \quad & \in  \quad  \bR \quad   \label{flow.lp2primal.c5} \\
&\forall  (\vv,\vv') \in E, & 1 \quad \ge \quad  d_{\vv\vv'} \quad & \ge \quad 0  \label{flow.lp2primal.c4}.
\end{align}
\end{subequations}
Since $f_j$ can be negative, we have added upper bounds of 1 on $d_{\vv\vv'}$ to prevent unboundedness of the objective.  In addition, $\bM$ is a large enough number forcing $d_{\va j}$ to zero, so variable $d_{\va j}$ has been eliminated from \eqref{flow.lp2primal.c2}.  We show that \eqref{flow.lp2primal} is a valid formulation for  $\min_{\vx \in \bB}f^b(\vx)$.

 \begin{theorem}
 \label{thm.perfer}
     The optimal value of \eqref{flow.lp2primal} equals $\min_{\vx\in\bB} f^b(\vx)$. Moreover, \eqref{flow.lp2primal} has an optimal solution satisfying $u_j=d_{j\sr}$ for all $j\in\cN$.
 \end{theorem}
 \begin{proof}{Proof.}
 For a fixed $\vd$, there exists $\vvu$ satisfying \eqref{flow.lp2primal.c1}--\eqref{flow.lp2primal.c3} if and only if
 \[
     d_{\mysl\va}-1+d_{j\sr}\ge 0
     \qquad \forall \va\in A,\; j\in\supp(\va).
 \]
 Necessity follows by combining $d_{\mysl\va}-1+u_\va\ge 0$, $u_j\ge u_\va$, and $d_{j\sr}\ge u_j$. For sufficiency, set $u_\va=1-d_{\mysl\va}$ for all $\va\in A$ and $u_j=d_{j\sr}$ for all $j\in\cN$. Hence projecting \eqref{flow.lp2primal} onto the $\vd$ variables gives \eqref{flow.lp3primal} below, and the two formulations have the same objective value because the objective depends only on $\vd$. The constraint matrix of \eqref{flow.lp3primal} is the node-edge incidence matrix of a bipartite graph with node sets $A$ and $\cN$, together with bound constraints. Therefore it is totally unimodular, and \eqref{flow.lp3primal} has an integral optimal solution. For any $\vx\in\bB$, setting $d_{j\sr}=x_j$ and $d_{\mysl\va}=1-\vx^\va$ gives a feasible solution of \eqref{flow.lp3primal} with objective value $f^b(\vx)$. Conversely, for any binary feasible $\vd$, define $x_j=d_{j\sr}$. Since $-f_\va\ge 0$, we may decrease $d_{\mysl\va}$, without increasing the objective, until $d_{\mysl\va}=1-\prod_{j\in\supp(\va)}x_j$. The objective value is then $f^b(\vx)$. Thus the optimal value of \eqref{flow.lp2primal} equals $\min_{\vx\in\bB}f^b(\vx)$. The sufficiency construction above also shows that an optimal solution can be chosen with $u_j=d_{j\sr}$ for all $j\in\cN$.
 \Halmos
 \end{proof}

 Thus, the projection of \eqref{flow.lp2primal} onto the $\vd$ variables is an integral polytope, and its optimal value matches that of $\min_{\vx \in \bB}f^b(\vx)$. Next, we apply the Fourier-Motzkin elimination to aggregate \eqref{flow.lp2primal.c1}, \eqref{flow.lp2primal.c2}, \eqref{flow.lp2primal.c3} such that free $u_{\va}$ and $ u_{j}$ are projected out. Then, \eqref{flow.lp2primal} simplifies to a projected formulation as follows:
  \begin{subequations}
\label{flow.lp3primal}
\begin{align}
\minc(G^b) ~\deq ~&\min \quad  & \sum_{\va \in A} (-f_{\va} d_{\mysl\va}) + \sum_{j \in \cN} f_j d_{j \sr} & \label{flow.lp3primal.obj}  \\
\text{s.t.}\; & \forall  \va \in A, j \in  \supp(\va), & d_{\mysl\va} -  1 +  d_{j\sr} \quad &\ge \quad 0  \label{flow.lp3primal.c1}   \\
&\forall  \va \in A, &  1 \quad \ge \quad d_{\mysl\va}  \quad & \ge \quad 0  \label{flow.lp3primal.c7}\\
&\forall j \in \cN, &   1 \quad \ge \quad d_{j \sr}  \quad & \ge \quad 0  \label{flow.lp3primal.c6}.
\end{align}
\end{subequations}

Given that $-f_{\va}$ is non-negative, the upper bound constraints $1 \ge d_{\mysl\va}$ are redundant and removable.
By adding back the subtracted constant $\sum_{\va \in A} f_{\va} + f_{\vz}$ in \eqref{eq.sep.fb} to $\minc(G^{b})$, we obtain:

  \begin{subequations}
\label{flow.lp4primal}
\begin{align}
\minc(G) ~\deq~  &\min \quad  &\sum_{\va \in A} f_{\va} + f_{\vz} + \sum_{\va \in A} (-f_{\va} d_{\mysl\va}) + \sum_{j \in \cN} f_j d_{j \sr} & \label{flow.lp4primal.obj}  \\
\text{s.t.}\; & \forall  \va \in A, j \in  \supp(\va), & d_{\mysl\va} -  1 +  d_{j\sr} \quad &\ge \quad 0  \label{flow.lp4primal.c1}   \\
&\forall  \va \in A, &   d_{\mysl\va}  \quad & \ge \quad 0  \label{flow.lp4primal.c7}\\
&\forall j \in \cN, &   1 \quad \ge \quad d_{j \sr}  \quad & \ge \quad 0  \label{flow.lp4primal.c6}.
\end{align}
\end{subequations}

The dual problem of \eqref{flow.lp4primal} is as follows:
\begin{subequations}
\label{flow.lp3dual}
\begin{align}
\maxf(G) ~\deq~ &\max \quad  & \sum_{\va \in A} f_{\va} + f_{\vz} +  \sum_{\va \in A} \sum_{j \in \supp(\va)} \rho_{\va j}  - \sum_{ j \in \cN} \beta_j & \label{flow.lp3dual.obj} \\[2mm]
\text{s.t.}\;
& \forall \va \in A,
&   \sum_{j \in \supp(\va)} \rho_{\va j} \quad & \le \quad -f_{\va} & \label{flow.lp3dual.c1} \\
& \forall j \in \cN,
& \sum_{\va \in A: j \in \supp(\va)} \rho_{\va j} - \beta_j \quad &   \le \quad f_{j} & \label{flow.lp3dual.c2} \\
& \forall j \in \cN,
&  \beta_j \quad & \ge \quad 0. & \label{flow.lp3dual.c3} \\
& \forall \va \in A, j \in \supp(\va),
&  \rho_{\va j} \quad & \ge \quad 0. & \label{flow.lp3dual.c4}
\end{align}
\end{subequations}

Enforcing non-negativity on the objective of the dual problem \eqref{flow.lp3dual}, we obtain  the final linear inequality system on the coefficients of $f$:
\begin{equation}
\label{flow.extended}
    \sum_{\va \in A} f_{\va} + f_{\vz} +  \sum_{\va \in A} \sum_{j \in \supp(\va)} \rho_{\va j}   - \sum_{ j \in \cN} \beta_j \ge 0, \quad \eqref{flow.lp3dual.c1}, \eqref{flow.lp3dual.c2}, \eqref{flow.lp3dual.c3}, \eqref{flow.lp3dual.c4}.
\end{equation}
This system defines an extended formulation of the non-negativity constraint for $f$.

\begin{theorem}
\label{thm.extend}
    $\min\limits_{\vx \in \bB}f(\vx) = \minc(G) =   \maxf (G)$. Moreover,  the non-negativity constraint $\min\limits_{\vx \in \bB}f(\vx) \ge 0$ has an  extended formulation  \eqref{flow.extended} of size $\bO(md)$.
\end{theorem}
\begin{proof}{Proof.}
Note that $f(\vx)=f^b(\vx) + \sum_{\va \in A} f_{\va} + f_{\vz}$.
It follows from the duality between $ \minc(G) $ and $   \maxf(G)$ and \Cref{thm.perfer} that $\min\limits_{\vx \in \bB}f(\vx) = \minc(G) =   \maxf(G)$. Then, the results follow.
  \Halmos
\end{proof}

\begin{remark}
\label{remark1}
We recall another derivation of the proposed extended formulation for $\min_{\vx \in \bB}f(\vx)$ from the continuous extension. Note that
 every binary polynomial admits a continuous multilinear extension obtained by extending its domain to the unit hypercube, and the convex envelopes of the binary polynomial and its extension coincide (Remark~1.3 of \cite{rikun1997convex}).  Consider now the NNS polynomial $f= f_{\vz} +\sum_{\va \in A} f_{\va} \vx^{\va} + \sum_{j \in \cN} f_j x_j$, and denote the convex envelope of $f(\vx)$ (and its continuous extension)  by  $\env(f)(\vx)$. It follows from Example 5 of \cite{meyer2005convex} (or Corollary 3.11 of \cite{tawarmalani2013explicit}) that  $\env(f)(\vx)=f_{\vz} +\sum_{\va \in A} \max_{j \in \supp(\va)} f_{\va} x_j + \sum_{j \in \cN} f_j x_j$, so $\min_{\vx \in \bB}f(\vx) = \min_{\vx \in [0,1]^n}\env(f)(\vx)$.  The envelope then leads to the formulation \eqref{flow.lp4primal}.
\end{remark}

\section{Concave extensions of  PS polynomials.}
\label{sec.concaveextensions}
We have introduced the polyhedral description for the non-negative NNS polynomials, and PS polynomials are the ``negation'' of NNS polynomials. In this section, we parameterize two \emph{piecewise linear concave extensions}  of PS polynomials. Since we take the set of binary monomials as the basis of $\bR(\vx)$, any binary polynomial $f(\vx) \in \bR(\vx)$ is isomorphic to a vector $f \in \bR^{\bB}$, whose entries are coefficients of $f(\vx)$. Similar to the vectorization of circuit polynomials \cite{dressler2017positivstellensatz}, we view binary polynomials as vectors in $\bR^{\bB}$. This leads to our parameterization of the concave extensions as ensembles of linear maps on PS polynomials; \ie the affine components of the concave extensions depend linearly on the coefficients of $f$.

\begin{remark}
The concave (resp. convex) envelopes of supermodular  (resp. submodular) functions already appeared in \cite{wolsey1999integer}. The envelopes have been exploited for various problem types
\cite{atamturk2022submodular,bach2019submodular} in mathematical optimization, especially for continuous multilinear polynomials \cite{nguyen2013deriving,tawarmalani2013explicit}.   In this section, the goal is to show that two types of existing  concave extensions for PS polynomials may have different representation sizes (parameterizations), and the size of one parameterization admits further reduction.
\end{remark}

Every linear map in $\bR(\vx)$ can be represented by a matrix $M \in \bR^{\bB \times \bB}$ that transforms any binary polynomial $f(\vx)=\sum_{\va \in \bB} f_\va x^\va$ to another polynomial $(Mf)(\vx) \deq \sum_{\va' \in \bB}\left(\sum_{\va \in \bB} M[\va',\va]f_\va\right)\vx^{\va'}$, where $M[\va',\va]$ is the $(\va',\va)$-th entry of the matrix $M$.
We consider  linear maps (\ie matrices) that transform  PS polynomials into linear polynomials.

\begin{definition}
\label{def.overtransset}
    Given $\myvs \in \{-1,0,1\}^{\bB}$ such that $s_{\bB_{0:1}} = \vz$ and $s_{\bB_{2:n}} \ge \vz$, a set $\cM(\myvs)$ of matrices in $\bR^{\bB \times \bB}$  is exact w.r.t. $\myvs$, if, for every PS polynomial $f$ signed-supported within $\myvs$ and every $\vx \in \bB$, $\min\limits_{M \in \cM(\myvs)}(Mf)(\vx) = f(\vx)$ holds.
\end{definition}

We call matrices $M$ in  $\cM(\myvs)$ \emph{overestimation} matrices, since, for any  polynomial $f$  signed-supported within $\myvs$, the linear polynomial $Mf$ overestimates $f$ on $\bB$.  We find that $\min\limits_{M \in \cM(\myvs)}(Mf)(\vx)$ is a \emph{piecewise linear concave extension} of $f(\vx)$, because it can recover exactly the evaluation of $f$ on $\bB$. We call $\cM(\myvs)$  \emph{minimal}, if no proper subset of $\cM(\myvs)$ is exact w.r.t. $\myvs$. Note that overestimation matrices are sparse in the sense that their non-zero entries are determined by the signed support patterns of PS  polynomials.

 \subsection{Standard extension.}

The first concave extension, called  \emph{standard extension} (or paved extension) (see \citet{crama1993concave}), is based on the standard linearization method  for binary  monomials (see \citet{glover1973further}). We next express the standard concave extension using overestimation matrices.

Consider the family $\Sigma(\myvs)$ of all maps $\sigma:  \supp(\myvs) \to \cN, \va \mapsto \sigma(\va) =j$ such that $j \in \supp(\va)$ holds. Thus,  each $\sigma \in \Sigma(\myvs)$ defines a linearization of $\vx^\va$ as $\vx^{\veu_{\sigma(\va)}} =x_{\sigma(\va)}=  x_j$. The corresponding overestimation matrix $M_\sigma$ is defined as follows. For every $\va \in \supp(\myvs)$, we set the entries $M_\sigma[\veu_{\sigma(\va)},\va]$ to 1; and set all other entries to 0. Let $\sigma^{-1}(j)$ denote the preimage of $j$ under $\sigma$.
Consequently, we have
 \begin{equation}
 \label{eq.multilinear}
 (M_{\sigma}f)(\vx) = \sum_{j \in \cN} (M_{\sigma} f)_{\veu_j} x_j  =   \sum_{j \in \cN} \big(\sum_{\va \in \sigma^{-1}(j)} f_\va\big) x_j  =  \sum_{\va \in \supp(\myvs)} f_\va x_{\sigma(\va)}.
 \end{equation}

The standard extension of $f$ is defined as $\min\limits_{\sigma \in \Sigma(\myvs)} M_{\sigma}f$, and the associated set of overestimation matrices is exact.

 \begin{proposition}
 For every $\sigma \in \Sigma(\myvs)$, $M_{\sigma}$ is an overestimation matrix w.r.t. $\myvs$. In addition,  $\{M_{\sigma}\}_{\sigma \in \Sigma(\myvs)}$ is exact w.r.t.  $\myvs$.
 \end{proposition}
  \begin{proof}{Proof.}
Let $f \in \bR(\vx)$ be a PS polynomial signed-supported within $\myvs$. Note that it has no linear terms, and $f_\va \geq 0$ for every $\va \in \supp(\myvs)$. Thus, for every $\vx \in \bB$, $\vx^\va = \prod_{j \in \supp(\va)} x_j  = \min\limits_{j \in \supp(\va)} x_j \le \vx^{\veu_{\sigma(\va)}} = x_{\sigma(\va)}$.  This implies that $f(\vx) = \sum_{\va \in \supp(\myvs)} f_\va \vx^\va  \le \sum_{\va \in \supp(\myvs)} f_\va x_{\sigma(\va)} =  (M_{\sigma}f)(\vx)
$ on $\bB$. Consequently,
 $M_{\sigma}$  is an overestimation matrix w.r.t. $\myvs$. To prove that $\{M_{\sigma}\}_{\sigma \in \Sigma(\myvs)}$  is exact w.r.t. $\myvs$, let $\vx $ be an arbitrary point in $\bB$. Since $\supp(f) \subseteq \supp(\myvs)$, choose $\sigma' \in \Sigma(\myvs)$ such that, for every $\va \in \supp(\myvs)$, $\sigma'(\va) = \argmin_{j \in \supp(\va)} x_j$. Then, for every $\va \in \supp(\myvs)$,
\[
\vx^{\va}=\prod_{j\in\supp(\va)}x_j
=\min_{j\in\supp(\va)}x_j
=x_{\sigma'(\va)}
=\vx^{\veu_{\sigma'(\va)}} ,
\]
which implies $f(\vx) = (M_{\sigma'}f)(\vx)$. Therefore, the result follows.
 \Halmos \end{proof}

\subsection{Lovász extension.}
The second concave extension,  called   \emph{Lovász extension}, is based on the  concave envelope of supermodular functions (see \citet{lovasz1983submodular}).  Treating the PS polynomial as a supermodular function, we can construct its facets for any dimension $n$ explicitly.  We next express the facets as linear maps.

Consider the set $\Pi$ of all permutations on elements of $\cN$. We denote by $\veu_{\pi([j])}$ the sum $\sum_{i \in [j]} \veu_{\pi(i)}$ of unit vectors. Note that $\veu_{{\pi}([j])}$ is in the binary hypercube $\bB$, and the permutation $\pi$ corresponds to a path in $\bB$ from the all-zero vector $\vz$  to the all-one vector $\veu$, where intermediate points  $\veu_{{\pi}([j])}$ follow   $\veu_{{\pi}([j-1])}$ ($j \in [2:n]$) sequentially. Along this path,  we evaluate the supermodular function $f$ at these points. When $x_{\pi(j)}$ increases from 0 to 1, using the finite-difference method, we compute the ``discrete'' partial derivative of $f\big(\veu_{{\pi}([j])}\big) - f\big(\veu_{{\pi}([j-1])}\big)$  w.r.t. the variable $x_{\pi(j)}$. The partial derivative is the sum of some $f_\va$, for which the supports of exponents  $\va$ are subsets of   $\pi([j]) $ and contain $\pi(j)$, and the set of these exponents is denoted as $A^j(\pi) \deq \{\va \in \bB: \pi(j) \in \supp(\va) \subseteq \pi([j]) \}$.

For each $\pi \in \Pi$, we construct a matrix  $M_\pi \in \{0,1\}^{\bB \times \bB}$ as follows.
For every $j \in \cN$ and
$\va \in A^j(\pi)$,  we set $M_\pi[\veu_{\pi(j)},\va]$ to 1; and we set all other entries to zero. Consequently, we have
  \begin{equation}
 \label{eq.suplinear}
 \begin{aligned}
 (M_\pi f)(\vx) &= \sum_{j \in \cN}  (M_\pi f)_{\veu_{\pi(j)}} x_{\pi(j)} = \sum_{j \in \cN}  \big(\sum_{\va \in A^j(\pi)} f_\va\big) x_{\pi(j)}\\
 &=  \sum_{j \in \cN}  \big(f\big(\sum_{ i \in [j]}\veu_{\pi(i)}\big) - f\big(\sum_{ i \in [j-1]}\veu_{\pi(i)}\big)\big) x_{\pi(j)}.
  \end{aligned}
 \end{equation}

 \begin{proposition}
For every $\pi \in  \Pi$, $M_\pi$ is an overestimation matrix w.r.t. $\myvs$. In addition,  $\{M_\pi\}_{\pi \in \Pi}$ is exact w.r.t. $\myvs$.
 \end{proposition}
 \begin{proof}{Proof.}
 It follows from \Cref{prop.subchar} that the PS polynomial $f$ is supermodular.
 According to \citet{xu2023}, any supermodular function has a concave piecewise linear extension. Moreover, each linear piece admits exactly  the form of the last equation in \eqref{eq.suplinear}.  Therefore, $(M_\pi f)(\vx) \ge f(\vx)$ on $\bB$, and $M_\pi$ is an overestimation  matrix w.r.t. $\myvs$. Furthermore, for every $\vx \in \bB$, there exists $\pi' \in \Pi$ such that $ f(\vx) =  \sum_{j \in \cN}  \big(f\big(\sum_{ i \in [j]}\veu_{\pi'(i)}\big) - f\big(\sum_{ i \in [j-1]}\veu_{\pi'(i)}\big)\big) x_{\pi'(j)} = (M_{\pi'} f)(\vx)$; see \citet{xu2023}. Thus, $\{M_\pi\}_{\pi \in \Pi}$ is exact w.r.t. $\myvs$.
 \Halmos \end{proof}

 \subsection{Comparison of concave extensions.}
 \label{s:concext}
 We compare  the sizes of proposed concave extensions, which are the dominant factors contributing to the size of the proposed relaxations for BPO. It follows from \Cref{def.overtransset} that an overestimation matrix has its non-zeros indexed by rows $\{\veu_{j'}\}_{j' \in \cN}$ (corresponding to linear monomials) and columns indexed by $\supp(\myvs)$ (corresponding to monomials of $f$).

 Let $m \deq |\myvs|$ be the number of monomials indexed by $\supp(\myvs)$. Let $d \deq \max_{\va \in \supp(\myvs)}|\va|$ be the maximum degree of PS polynomials signed-supported within $\myvs$. Recall that $n = |\cN|$.  Let $\cM(\myvs)$ be an exact set of overestimation matrices w.r.t $\myvs$ in \Cref{def.overtransset}. Recall that
 the cardinality of $\cM(\myvs)$ is the number of matrices required to extend any  PS polynomial signed-supported within $\myvs$. Therefore, we are primarily interested in the cardinality of the two extension methods. Let $\Gamma(\myvs)$ denote  an upper bound to  the minimum  cardinality of the exact set w.r.t. $\myvs$, which we call the \emph{concave-extension complexity upper bound} w.r.t. $\myvs$.   For the standard extension approach, the cardinality $|\cM(\myvs)|$ is $|\Sigma(\myvs)|=\prod_{\va \in \supp(\myvs)} |\va|$ with an upper bound  $d^{m}$. For the Lovász extension approach, the cardinality $|\cM(\myvs)|$ is $|\Pi|= n!$ with an upper bound $n^{n}$.

  We only need exact concave extensions that consist of the fewest overestimation matrices possible. However, the Lovász extension of a sparse PS polynomial $f$, as its concave envelope, may contain redundant facets.  In degenerate cases, some facets may become the same, \eg facets for affine polynomials.  We remark that the Lovász extension relates every linear overestimator $M_\pi f$ to its \emph{active binary points} $\vx \in \bB$ (in the aforementioned path induced by $\pi$), for which $(M_\pi f)(\vx) = f(\vx)$. We require only one facet to be active at each binary point; however, a single facet can be active at multiple points, which is the source of redundancy.
  As shown in \Cref{exampl.extension}, we can relax the   Lovász extension to reduce the number of overestimation matrices.

\begin{example}
\label{exampl.extension}
Let $f(\vx) = x_1 x_2  + 2 x_2 x_3 x_4  + 5x_4 x_5$ over 5 unknowns. The standard extension creates 2, 3, and 2 linear polynomials  overestimating $x_1 x_2, 2x_2 x_3 x_4$, and $ 5x_4 x_5$, respectively. For example, $5 x_4 x_5$ is overestimated by $5x_4$ and $5x_5$. Each overestimation matrix corresponds to a combination of these polynomials, thus, the number of  these matrices  is $12 = 2 \times 3  \times 2$. The  Lovász extension has $120 = 5!$ facets, regardless of the sparsity of $f$. Then, it needs 120  overestimation matrices.
\end{example}

 We propose a \emph{filtering method} that outputs a relaxed Lovász extension for $f$ with a better bound on $|\cM(\myvs)|$. Starting with $\cM(\myvs) = \varnothing$, the method has the  following steps:
 \begin{enumerate}
     \item the method finds a point  $\vx \in \{0,1\}^{\cN}$ such that  there is no overestimation matrix $M_\pi \in \cM(\myvs)$ satisfying $(M_\pi f)(\vx) = f(\vx)$, if such a point does not exist, then the method terminates;
     \item then, using the exact relation between facets and binary points, the method adds a new overestimator $M_\pi f$ exact at $\vx$ and goes back to step 1.
 \end{enumerate}
  By the pigeonhole principle, the method outputs an exact set of overestimation matrices, whose cardinality is at most $2^{n}$. Finally, the output set is guaranteed to be  \emph{minimal}, \ie there is no proper subset of the output set that results in a concave extension. Otherwise, the filtering method could have removed redundant overestimators.  We obtain the following result on the concave-extension complexity upper bound.

\begin{corollary}\label{prop.filter}
Given  $\myvs \in \{-1,0,1\}^{\bB}$,
      $ \Gamma(\myvs) \le \min(2^{n},  d^{m})$, where $2^{n}$ (resp. $d^{m}$) is achievable by the relaxed Lovász (resp. standard) extension.
\end{corollary}

The standard extension is oriented towards the monomials, as it can extend any  PS polynomial with a specified set of monomials. It is particularly useful when  the polynomial is sparse. In principle, one can apply a similar filtering method to remove redundant overestimation matrices for the standard extension. On the other hand,  the relaxed Lovász extension is oriented towards the variables, as it can extend any $n$-variate  PS polynomial. Let the following subset of $\cN$ index  actual variables appearing in the monomials of the PS polynomial:
\begin{equation}
\label{eq.cns}
    \cN(\myvs)\deq\bigcup_{\va \in \supp(\myvs)}  \supp(\va).
\end{equation}

\begin{remark}
The Lovász extension also provides the convex envelope $\env(f)$ for any (submodular) NNS polynomial $f$. \Cref{remark1} shows that the convex envelope of an NNS polynomial coincides with its standard extension. This coincidence also holds for PS polynomials; however, the parameterizations of the two extensions differ, since the Lovász extension depends solely on the polynomial’s values. Moreover, there exists a straightforward way to reduce the complexity of the concave-extension in the Lovász construction.
\end{remark}

\section{BNP and BPO for NDS polynomials.}
\label{sec.lifting}

  In this section, we characterize the set of binary non-negative NDS polynomials, and propose a novel reformulation of BPO for these polynomials. We have shown that a given  polynomial has a unique signed support decomposition and its PS component has   piecewise linear concave extensions. As NNS polynomials can incorporate linear terms, we leverage concave extensions to reformulate BNP for the polynomial into BNPs for multiple NNS polynomials. This leads to the reformulation of BPO for that polynomial. We first present a necessary and sufficient condition  for the  binary non-negativity.

\begin{lemma}\label{lem.nonneg2}
 For every binary polynomial $f \in \bR(\vx)$, let $\myvs' \deq \ssv(\ppn(f))$, and let $\cM(\myvs')$ be an exact set of overestimation matrices w.r.t. $\myvs'$. Then $f$ is binary non-negative if and only if for every $M \in \cM(\myvs')$,  the NNS polynomial $ \nn(f) + M\ppn(f)$ is binary non-negative.
\end{lemma}
\begin{proof}{Proof.}
Let $M \in \cM(\myvs')$.
Since $\ppn(f)$ is  NPS and $M$ is an overestimation matrix w.r.t. $\myvs'$, $M\ppn(f)$ is a linear polynomial. Since $ \nn(f)$ is NNS, $ \nn(f)+ M \ppn(f)$  is NNS as well. The exactness of $\cM(\myvs')$ means that  $\ppn(f)(\vx) = \min\limits_{M \in \cM(\myvs')} (M\ppn(f))(\vx)$ for every $\vx \in \bB$. It follows that $f$ equals $ \min\limits_{M \in \cM(\myvs')}  (\nn(f) + M\ppn(f))$  on  $\bB$. Thus, the binary non-negativity condition $\min\limits_{\vx \in \bB}f(\vx) \ge 0$ holds, if and only if, for every $M \in \cM(\myvs')$,  the NNS polynomial $ \nn(f) + M\ppn(f)$ is binary non-negative.
\Halmos \end{proof}

We note that the NNS polynomial $\nn(f) + M\ppn(f)$ is an overestimator for the binary polynomial $f$ on $\bB$. Moreover, the BNP for $f$ is equivalent to the BNPs for all its NNS overestimating polynomials induced by $\cM(\myvs')$.   We also have a signed support decomposition of signed support vectors,  as a counterpart to the decomposition of binary polynomials.

\begin{definition}
\label{assm}
Given  $\myvs \in \{-1,0,1\}^{\bB}$ with $\bB_{0:1} \subseteq \supp(\myvs)$, its signed support decomposition $\myvs = \myvs^{\snn} + \myvs^{\spn}$ satisfies the following: $\myvs^{\snn}, \myvs^{\spn} \in  \{-1,0,1\}^{\bB}, s^{\snn}_{\bB_{2:n}} \le \vz, s^{\spn}_{\bB_{2:n}} \ge \vz, s^{\spn}_{\bB_{0:1}} = \vz$, and $ \supp(\myvs^{\snn}) \cap \supp(\myvs^{\spn}) = \varnothing$.
The derived parameters of $\myvs^i$ ($i \in \{\snn,\spn\}$) are  $m_i\deq|\myvs^i|, d_i\deq \max_{\va \in \supp(\myvs^i) }|\va|$, $n_i\deq|\cN(\myvs^i)|$. The derived parameters of $\myvs$ include those of $\myvs^{\snn},\myvs^{\spn}$ and $m\deq m_{{\snn}} +m_{{\spn}}$.
\end{definition}

In  the above definition, if the binary polynomial $f$ is signed-supported within $\myvs$, then it is an NDS polynomial. The signed support vectors $\myvs^{\snn},\myvs^{\spn}$ correspond to  $\nn(f),\ppn(f)$, respectively, \ie $\ssv(\nn(f)) \preceq \myvs^{\snn}$ and $\ssv(\ppn(f)) \preceq \myvs^{\spn}$.  More precisely, this implies that $\nn(f)$ is an $n_{{\snn}}$-variate ($n_{{\snn}} = n$) $d_{{\snn}}$-degree binary polynomial with $m_{{\snn}}$ monomials, and $\ppn(f) $ is an $n_{{\spn}}$-variate $d_{{\spn}}$-degree binary polynomial with $m_{{\spn}}$ monomials. We will use these parameters to quantify the complexities of  BNP and BPO for $f$. We next give the formulation of its BNP.
\begin{definition}
\label{defn:nm}
Given  $\myvs \in \{-1,0,1\}^{\bB}$ with $\bB_{0:1} \subseteq \supp(\myvs)$, let $\cM(\myvs^{\spn})$ be an exact  set of overestimation matrices w.r.t. $\myvs^{\spn}$. The set $\nm(\myvs)$ of binary non-negative NDS polynomials w.r.t. $\myvs$ is
\begin{multline}
    \label{eq.nm}
      \{f \in \bR(\vx):
\nn(f) \in \ssc(\myvs^{\snn}) \land \; \ppn(f) \in \ssc(\myvs^{\spn}) \land \\ \forall M \in \cM(\myvs^{\spn})   \; \nn(f) + M\ppn(f) \in \nnn(\myvs^{\snn})  \}.
\end{multline}

\end{definition}

The above definition explains the necessity of the condition $\bB_{0:1} \subseteq \supp(\myvs)$ in \Cref{assm}. This condition ensures that, after applying an  overestimation matrix $M$, the inclusions $\supp(M\ppn(f)) \subseteq \bB_{0:1} \subseteq \supp(\myvs^{\snn})$ hold. Thus, $\ssv(\nn(f) + M\ppn(f)) \preceq \myvs^{\snn}$, and \eqref{eq.nm} is well-defined.
 We  observe that  \eqref{eq.nm}
 has an alternative extended representation:
\begin{equation}
\label{eq.nmalter}
    \exists f^1,f^2 \in \bR(\vx)\; f^1 \in \ssc(\myvs^{\snn}) \land\ f^2 \in \ssc(\myvs^{\spn}) \land\ f = f^1 + f^2 \land
    \forall M \in \cM(\myvs^{\spn}) \; f^1 + Mf^2 \in \nnn(\myvs^{\snn}),
\end{equation}
where the signed support constraints can be translated into simple constraints on the coefficients of $f$.
The above system has the following properties.

\begin{theorem}\label{thm.nn}
$\nm(\myvs)$ is a convex polyhedral cone, and it has an extended LP formulation of $m$ monomial variables subject to a signed support constraint and  $\bO(\Gamma(\myvs^{\spn}))$  non-negativity constraints. Moreover, every polynomial $f$ signed-supported within $\myvs$ is binary non-negative  if and only if  $ f \in \nm(\myvs)$.
\end{theorem}
 \begin{proof}{Proof.}
Looking into \eqref{eq.nmalter}, the signed support constraints on $f^1$ and $f^2$ can be translated into a signed support constraint on $f$ itself, and we need $m=m_{{\snn}} +m_{{\spn}}$ monomial variables to represent $f$. As $|\cM(\myvs^{\spn})| \le \Gamma(\myvs^{\spn})$, there are at most $\bO(\Gamma(\myvs^{\spn}))$ cones $\nnn(\myvs^{\snn})$, each of which has a non-negativity constraint. Thus, $\nm(\myvs)$ is a convex polyhedron with an LP representation of the stated size.
Suppose $f$ is a binary polynomial signed-supported within $\myvs$. Since $ \bB_{0:1} \subseteq \supp(\myvs^{\snn})$, it follows that $\ssv(f^1 + Mf^2) \preceq \myvs^{\snn}$ for every $M \in \cM(\myvs^{\spn})$. It follows from \Cref{lem.nonneg2} and  \Cref{thm.subnn} that $f$ is binary non-negative  if and only if $f^1 + Mf^2 \in \nnn(\myvs^{\snn})$ is binary non-negative for every $M \in \cM(\myvs^{\spn})$. This condition is equivalent to $f \in \nm(\myvs)$. Note that for any $f \in \nm(\myvs)$, $\lambda f$ is non-negative for any positive scalar $\lambda$, so $\nm(\myvs)$ is a cone.
\Halmos \end{proof}

We next characterize the complexity of BNP for  any polynomial $f$ signed-supported within $\myvs$. The BNP  is equivalent to a separation problem for $\nm(\myvs)$, which asks to separate $f$ from the cone.  The separation complexity is implicitly parameterized by $\myvs$  as follows.

\begin{proposition}
 \label{propo.sepnn}
   There is an algorithm solving the separation problem for the convex polyhedral cone $\nm(\myvs)$ in $\bO(d_{{\snn}}m_{{\snn}}^2\Gamma(\myvs^{\spn}))$ time.
\end{proposition}
\begin{proof}{Proof.}
There are at most $\Gamma(\myvs^{\spn})$ sub-cones  $\nnn(\myvs^{\snn})$ in  $\nm(\myvs)$.  It follows from \Cref{assm} that $\bB_{0:1}\subseteq\supp(\myvs^{\snn})$, so $m_{{\snn}}$ counts the constant and linear monomials and hence $n\le m_{{\snn}}-1=\bO(m_{{\snn}})$. It suffices to search for a sub-cone $\nnn(\myvs^{\snn})$ that does not contain the target polynomial $\nn(f) + M\ppn(f)$   and separate it from that cone. By \Cref{thm.sepsub}, each such separation costs $\bO(d_{{\snn}}m_{{\snn}}^2)$ time, and the result follows.
\Halmos \end{proof}

Then, we want to solve BPO for $f$:
\begin{equation}
\label{eq.prim}
    \lambda^\star \deq \min_{\vx \in \bB} f(\vx).
\end{equation}
The BPO is equivalent to the dual problem $
    \lambda^\star  = \max_{\lambda \in \bR}\{\lambda: f - \lambda \textup{ is binary non-negative}\} $.
 Since $\lambda$ is a scalar, by \Cref{assm},  $\ssv(f - \lambda)  \preceq \myvs$. As every binary non-negative $f - \lambda$ must be in $\nm(\myvs)$, the dual problem  is equivalent to a  linear optimization on the convex cone $\nm(\myvs)$, which we call  \emph{signed reformulation} of BPO:
\begin{equation}
\label{eq.dualref}
    \lambda^\star  = \max_{\lambda \in \bR}\{\lambda: f - \lambda \in \nm(\myvs)\}.
\end{equation}

The signed reformulation can be solved using the ellipsoid algorithm (see \citet{grotschel1981ellipsoid}),
whose complexity is polynomial in the separation complexity bound $\bO(d_{{\snn}}m_{{\snn}}^2\Gamma(\myvs^{\spn}))$ given in \Cref{propo.sepnn}. Moreover, it follows from \Cref{thm.extend} that  the signed reformulation admits an extended formulation of size $\bO(d_{{\snn}}m_{{\snn}}\Gamma(\myvs^{\spn})+md)$.

 The dominant factor in the complexity of solving the signed reformulation   is  the concave-extension complexity upper bound  $\Gamma(\myvs^{\spn})$, which determines the achievable number of  sub-cones $\nnn(\myvs^{\snn})$ in  $\nm(\myvs)$.
By \Cref{prop.filter}, this factor $\Gamma(\myvs^{\spn})$ grows exponentially w.r.t. the derived parameters of $\myvs^{\spn}$. We  illustrate the signed reformulation through the following example.

\begin{example}
\label{exampl.decomp}
Consider $f^1(\vx) =- c_1 x_2 x_3 -  c_2  x_1 x_3 x_4 - c_3 x_3 x_5 + c_4 x_2 + c_5 x_3- c_6 x_4 + c_7$ and $f^2(\vx) = c_8 x_1 x_2  + c_9 x_2 x_3 x_4  + c_{10} x_4 x_5$, where all coefficients $c$ are positive. The family of binary polynomials $f=f^1+f^2$ includes the specific polynomial in \Cref{exampl.pre}.  We use the standard extension method. For example, there is an overestimation matrix $M$ in $\cM(\myvs^{\spn})$ inducing a linear overestimator $(Mf^2)(\vx) = c_8 x_1 + c_9 x_4 + c_{10} x_4$ (using $c_8 x_1 \ge c_8 x_1 x_2, c_9 x_4 \ge c_9 x_2 x_3 x_4, c_{10} x_4\ge c_{10} x_4 x_5$). Thus, $f^1(\vx) + (Mf^2)(\vx)  =- c_1 x_2 x_3 -  c_2  x_1 x_3 x_4 - c_3 x_3 x_5 +  c_8 x_1 +c_4 x_2 + c_5 x_3 + ( c_9 + c_{10} - c_6 ) x_4  + c_7 $. There are in total 12 overestimation matrices in $\cM(\myvs^{\spn})$. Thus, we have 12  $\nnn(\myvs^{\snn})$ cones.
\end{example}

We examine the Lasserre hierarchy for BPO to illustrate how inner approximations based on binary non-negativity certificates lead to a hierarchy of relaxations. The $i$-th level Lasserre relaxation takes the form $ \lambda^i_{sos}\deq\max_{\lambda \in \bR}\{\lambda: f - \lambda \in \sos^{i}(n)\},$ where $\sos^{i}(n)$ is the set of $n$-variate SoS  polynomials of degree at most $i$. The relaxation value $\lambda^i_{sos}$ is a lower bound for $\lambda^\star$.  In the worst case, it requires $i=n$ to make the Lasserre relaxation exact (i.e., $\lambda^n_{sos} = \lambda^\star$) (see \citet{lasserre2002explicit}). Since $\sos^{1}(n) \subseteq \cdots \subseteq \sos^{n}(n)$, their lower bounds are non-decreasing, \ie $\lambda^1_{sos} \le \cdots \le \lambda^n_{sos}$. Generally, one expects to obtain either an exact minimum value or a good lower bound at a level lower than $n$.

We compare the signed reformulation with the last level of Lasserre relaxations.
The set $\sos^{n}(n)$ has an SDP representation of size $\bO(2^n)$, which does not depend on sparsity. By \Cref{thm.extend}, the LP representing $\nm(\myvs)$ has size $\bO(d_{{\snn}}m_{{\snn}}\Gamma(\myvs^{\spn})+md)$. Since $\Gamma(\myvs^{\spn})\le 2^n$ by \Cref{prop.filter}, this size is bounded by $\bO(d_{{\snn}}m_{{\snn}}2^n+md)$. Thus, up to polynomial factors, the worst-case size of the signed reformulation is comparable to that of the Lasserre hierarchy, while it can be much smaller depending on sparsity and signed support patterns. The signed reformulation will be the last level of relaxation in our hierarchies of relaxations, see \Cref{sec.refined} below.

\section{Refined signed support decomposition.}
\label{sec.refined}
For a binary polynomial, we refine the signed support decomposition to further exploit its signed support pattern. The refined decomposition uses sums of certain binary non-negative NDS polynomials, which we call \emph{signed certificates}. The refined decomposition allows us to construct nested families of conic inner approximations of $\nm(\myvs)$, which lead to  our new hierarchies of relaxations for BPO \eqref{eq.prim}.  We first define the refined signed support decomposition for signed support vectors.

\begin{definition}
\label{def.decomp}
Given $\myvs \in \{-1,0,1\}^{\bB}$,  a collection $ \{\vth^{k}\}_{k \in [\ell]}$ of signed support vectors is a refined signed support decomposition of $\myvs$, if   $\bB_{0:1} \subseteq \supp(\vth^k)$ and $\supp(\myvs) \subseteq \cup_{k \in [\ell]} \supp(\vth^k)$.
\end{definition}

For each $k \in [\ell]$, let $\vth^{k,{\snn}}$ and $\vth^{k,{\spn}}$ denote the signed support decomposition of $\vth^k$.
We then introduce the refined signed support decomposition that decomposes any binary polynomial $f$ signed-supported within $\myvs$ into a sum  of signed certificates and a PS polynomial.
We call the signed support vector of each summand a \emph{decomposed signed support vector}. The key idea is to choose the decomposed signed support vectors so that they form a refined signed support decomposition of $\ssv(f)$.

\begin{definition}
\label{def.cn}
Given $\myvs \in \{-1,0,1\}^{\bB}$,  let $\Theta(\myvs)\deq \{\vth^{k}\}_{k \in [\ell]}$ be a refined signed support decomposition of $\myvs$.  Then, the sum-of-binary-non-negative-NDS polynomials (for brevity, the sum of signed certificates)  w.r.t. $\Theta(\myvs)$ is the set $\snm(\Theta(\myvs))$:
{\small \begin{equation}
   \label{eq.snm}
  \{f \in \bR(\vx): f = g + \sum_{k \in [\ell]}  f^k \land\ \forall k \in [\ell]\; f^k \in \nm(\vth^k) \land\   g \ge \vz \land\ \supp(g)  \subseteq \supp(\myvs) \cup \supp(\bB_{0:1})\}.
\end{equation}}
Moreover, for each $k\in[\ell]$, let $\cM(\vth^{k,{\spn}})$ denote an exact set of overestimation matrices w.r.t. $\vth^{k,{\spn}}$ as in \Cref{def.overtransset}. These sets implement the cones $\nm(\vth^k)$ as in \Cref{defn:nm}; denote $\fL(\myvs)\deq\{\cM(\vth^{k,{\spn}})\}_{k\in[\ell]}$.
\end{definition}

 We look at  the set $ \snm(\Theta(\myvs))$ and its representation.  Any polynomial $f \in \snm(\Theta(\myvs))$ can be decomposed as the sum  of a PS polynomial $g$ and binary non-negative polynomials $f^k$ (signed certificates). This parameterization is flexible, as one can adjust the signed support vector $\vth^{k}$ of $f^k$ to decompose that of $f$.  Moreover, the sum of the concave-extension complexity upper bounds $\Gamma(\vth^{k,{\spn}})$ of  $f^k$ may be smaller than  $\Gamma(\myvs^{\spn})$ of $f$: this opens up possibilities for constructing relaxations for BPO with controllable complexity. We have the following observation.

 \begin{lemma}
 The set $\snm(\Theta(\myvs))$ in  \eqref{eq.snm} is a convex polyhedral cone, and
any $f \in   \snm(\Theta(\myvs))$ is binary non-negative.
 \end{lemma}
 \begin{proof}{Proof.}
Examining \eqref{eq.snm}, the set $\snm(\Theta(\myvs))$ is a Minkowski sum of the convex polyhedral cones $\nm(\vth^k)$ and $\{g: g \ge \vz \land\ \supp(g)  \subseteq \supp(\myvs) \cup \supp(\bB_{0:1})\}$, so it is a convex polyhedral cone.
 As $f^k$ and $g$ are binary non-negative, $f$ is binary non-negative as well.
\Halmos
 \end{proof}

Taking a closer look at $\snm(\Theta(\myvs))$ reveals that it encompasses two opposing categories of binary polynomials: $g$ and $f^k$.  They enable us to represent two extreme scenarios where $f$ can be either a PS or an NS polynomial.

In the following, we restrict the set $\Theta(\myvs)$ to satisfy the following additional conditions: (i) (NNS equivalence) for every $k \in [\ell]$, $\vth^{k,{\snn}} = \myvs^{\snn} $, (ii) (disjoint PS partition) for every $i, j \in [\ell]$, $\supp(\vth^{i,{\spn}}) \cap \supp(\vth^{j,{\spn}}) = \varnothing$, and  $\myvs^{\spn} = \sum_{k \in [\ell]} \vth^{k,{\spn}}$. We remark that these conditions are adjustable  based on the signed support pattern of $f$. For example, if $\supp(\myvs^{\snn}) = \varnothing$, we can set, for every $k \in [\ell]$, $\vth^{k,{\snn}} = \myvs$; for every $i, j \in [\ell]$, $\supp(\vth^{i,{\spn}}) \cap \supp(\vth^{j,{\spn}}) = \varnothing$, and  $\myvs = \sum_{k \in [\ell]} \vth^{k,{\spn}}$.

Since expanding the support $\vth^{k,{\snn}}$ of the NNS component of $\vth^k$ only incurs an additional computational cost polynomial in the derived parameters of $\vth^k$ (see \Cref{propo.sepnn} and \Cref{thm.sepsub}), this yields a computationally cheap way to enlarge the search space.  Thus, the NNS equivalence constraint stipulates that the support of  $\vth^{k,{\snn}}$ equals $\myvs^{\snn}$. Enlarging the support $\vth^{k,{\spn}}$ of the PS component of $\vth^k$ incurs a computational cost exponential in the derived parameters of $\vth^k$ (see \Cref{propo.sepnn} and \Cref{prop.filter}). Consequently, enlarging $\Theta(\myvs)$ in the latter manner is computationally expensive.  This explains why we make the supports of $\vth^{k,{\spn}}$ form a disjoint partition of $\myvs^{\spn}$ rather than merely a cover.

Based on the refined signed support decomposition, one can construct a convex relaxation for BPO, by replacing  $ \nm(\myvs)$ in the signed reformulation \eqref{eq.dualref} with its  conic inner approximation $\snm(\Theta(\myvs))$:
\begin{equation}
\label{eq.sumdual}
     \lambda^\star(\Theta(\myvs))  \deq \max_{\lambda \in \bR}\{\lambda: f - \lambda \in \snm(\Theta(\myvs))\}.
\end{equation}
Since $ \snm(\Theta(\myvs)) \subseteq \nm(\myvs)$, the value $\lambda^\star(\Theta(\myvs))$ is a lower bound to the optimal value $\lambda^\star$ of BPO.

Given $\myvs$, the \emph{support inference problem} asks:
\begin{multline}
\label{eq.inference}
    \text{for an acceptable complexity to compute } \lambda^\star(\Theta(\myvs))  \text{, determine the  decomposed  signed support} \\ \text{ vectors in } \Theta(\myvs) \text{ that yield the largest lower bound } \lambda^\star(\Theta(\myvs)).
\end{multline} This question can be viewed as an inverse problem, where we aim to infer the decomposed signed support vectors from observing the signed support vector of $f$.  The problem arises naturally when designing sparse non-negativity certificates to approximate a dense polynomial, e.g., generating circuit polynomials \cite{papp2023duality}.

\section{Hierarchies of relaxations for BPO.}
\label{sec.nest}
In this section, we tackle the challenging support inference problem \eqref{eq.inference} from the perspective of parameterized complexity. We adopt a hierarchical approach that takes advantage of the refined signed support decomposition to construct the family $\Theta(\myvs)$. Our approach involves gradually expanding the sets in $\Theta(\myvs)$ by incorporating more signed support vectors.  This leads to nested cones $\snm(\Theta(\myvs))$ of binary non-negative polynomials of increasing complexity. Based on \eqref{eq.sumdual}, we  construct two types of hierarchies of relaxations for BPO.

 We first introduce the template hierarchical partition, and then describe two applications of this partition to generate the family $\Theta(\myvs)$. Given a finite base set $C$, the output of the hierarchical partition is  a binary tree $T(C)$. The \emph{depth} of a node is the number of nodes on the path from that node to the root, including both endpoints; in particular, the root has depth $1$. Let $\bar{h}$ denote the maximum depth among all nodes, and let $r(C)$ denote the root node. The \emph{height} of a node at depth $h$ is defined as $\bar{h} + 1 - h$, so $\bar{h}$ is also the number of layers.  The $i$-th layer $V^i(C)$ of the tree comprises all nodes at height $i$, and we define its \emph{width} $\ell^i$ as $|V^i(C)|$. We write $V^{i}(C)=\{V^{i}_1(C),\ldots, V^{i}_{\ell^i}(C)\}$.

The tree construction starts from the first layer $V^1(C) \deq \{\{c\}\;|\; c \in C\}$ as a trivial partition and builds the layers $V^{i+1}(C)$ from $V^{i}(C)$ recursively, such that each layer forms a partition of $C$. If $\ell^i$ is even, then $V^{i}(C)$ is partitioned into pairs, and the union of two nodes in each pair is taken as their parent node in $V^{i+1}(C)$. If $\ell^i$ is odd, then a dummy node $\varnothing$ is added to $V^{i}(C)$, and the above procedure is applied as in the even case. When $\ell^i = 1$, the construction terminates with $V^{i}(C)=\{C\}$, and the root node is $r(C)=C$.  The tree $T(C)$ is called a \emph{partition tree} of $C$, and it has the following properties:

\begin{lemma}\label{prop.arb}
Given a finite set $C$, $r(C) = C$, and the maximum depth $\bar{h} \le \ceil{\log |C|}+1$. For every $i \in [\bar{h}]$,
$V^{i}(C)$ is a partition of $C$, $|V^i(C)| \le  2^{\bar{h}-i} $, and, for every $k \in [\ell^i],|V^i_k(C)| \le 2^{i - 1}$.
\end{lemma}
\begin{proof}{Proof.}
    We know that, at the leaves, $V^{1}(C)$ is a partition of $C$. Using a basic induction argument, we know that at each height $1 \le i \le \bar{h} - 1$,  as $V^{i}(C)$ is a partition of $C$ and the nodes of $V^{i+1}(C)$ are the union of two  nodes of  $V^{i}(C)$, $V^{i+1}(C)$ is also a partition of $C$.  Since the final layer has one node and remains a partition of $C$, its unique node is $C$; hence $r(C)=C$. Since, at each height, we take the union of two nodes to form a node of the next height, $T(C)$ is a binary tree.  This reduction is logarithmic. It follows that $\bar{h} \le  \ceil{\log |C|}+1$, $|V^i(C)| \le 2^{\bar{h}-i}$, and $|V^i_k(C)| \le 2^{i - 1} $.
\Halmos \end{proof}

 Note that we are only concerned with the upper bounds:  when $|C|$ is not a power of 2, other dummy nodes $\varnothing$ are added to simplify the counting  process and binary tree construction process.
The layers form a \emph{nested family} of partitions of $C$.
 The partition tree $T(C)$ can be constructed in time polynomial in $|C|$.   \Cref{fig.arb} shows a partition tree $T(C)$ for $C = \{c_1, \ldots, c_{2^h-1}\}$, where $h$ is a natural number and $\bar h=h+1$.

 \begin{figure}[!h]
\centering

\begin{tikzpicture}[level/.style={sibling distance=47mm/#1}]
\node [draw] (z){$\{c_1,\ldots, c_{2^h-1}\}$}
  child {node [draw] (a) {$\{c_1,\ldots, c_{2^{h-1}}\}$}
    child {node {$\vdots$}
      child {node [draw] (d) {$\{c_1\}$}}
        child {node [draw] (e) {$\{c_2\}$}}
    }
     child {node {$\vdots$}
    }
  }
  child {node [draw] (j) {$\{c_{2^{h-1}+1},\ldots, c_{2^h-1}\}$}
     child {node {$\vdots$}
    }
   child {node {$\vdots$}
    child {node [draw] (o) {$\{c_{2^h-1}\}$}}
    child {node [draw] (p) {$\varnothing$}
        child [grow=right] {node (q) {} edge from parent[draw=none]
          child [grow=right] {node [align=left] (u) {layer $V^1(C)$\\(height $1$, depth $h+1$)} edge from parent[draw=none]
            child [grow=up] {node (r) {$\vdots$} edge from parent[draw=none]
              child [grow=up] {node [align=left] (s) {layer $V^{h}(C)$\\(height $h$, depth $2$)} edge from parent[draw=none]
                child [grow=up] {node [align=left] (t) {layer $V^{h+1}$\\(height $h+1$, depth $1$)} edge from parent[draw=none]
                }
            }
          }
        }
      }
    }
  }
};
\path (o) -- (e) node (x) [midway] {$\cdots$};
\end{tikzpicture}
\caption{Partition tree $T(C)$}
\label{fig.arb}
\end{figure}

Next, we consider minimizing a binary polynomial $f$ signed-supported within the vector $\myvs$. We shall see that the hierarchical partition of a certain set (which we shall specify later) leads to refined signed support decompositions of $f$, which, in turn,  yield our template hierarchy of (signed) relaxations.  We associate the node indexed at height $i$ and $k \in [\ell^i]$ with a decomposed signed support vector $\vth^{i,k} =\vth^{i,k,{\snn}} + \vth^{i,k,{\spn}}\in \{-1,0,1\}^{\bB}$. Suppose that we can construct  an exact set $\cM(\vth^{i,k,{\spn}})$  of overestimation matrices w.r.t. $\vth^{i,k,{\spn}}$. This leads to a set $\nm(\vth^{i,k})$ of binary non-negative polynomials as in \Cref{defn:nm}. Given the families $\Theta^{i}(\myvs)\deq \{\vth^{i,k}\}_{k \in [\ell^i]},\fL^i(\myvs)\deq\{\cM(\vth^{i,k,{\spn}})\}_{k \in [\ell^i]}$ as in \Cref{def.cn},  we obtain the cone $\snm(\Theta^{i}(\myvs))$ of binary non-negative polynomials by taking the Minkowski sum of $ \nm(\vth^{i,k})$. Substituting $\snm(\Theta^{i}(\myvs))$ in \eqref{eq.sumdual} yields the following \emph{level-$i$ signed relaxation}:
\begin{equation}
\label{eq.templatesparserelax}
    \lambda^i  \deq \lambda^\star(\Theta^i(\myvs)) =  \max_{\lambda \in \bR}\{\lambda: f - \lambda \in \snm(\Theta^{i}(\myvs))\}.
\end{equation}
The  program searches for  a maximum  $\lambda$ such that $f - \lambda$ can be decomposed as a sum in $\snm(\Theta^{i}(\myvs))$. As any polynomial in $\snm(\Theta^{i}(\myvs))$ is binary non-negative,  $\lambda$ is a lower bound for BPO \eqref{eq.prim}.

We say that the family $\{\snm(\Theta^{i}(\myvs))\}_{i \in [\bar{h}]}$ is \emph{nested}, if
$\snm(\Theta^1(\myvs))   \subseteq  \cdots  \subseteq\snm(\Theta^{\bar{h}}(\myvs))$; the family is \emph{complete}, if $ \nm(\myvs) \subseteq  \snm(\Theta^{\bar{h}}(\myvs))$.

\begin{lemma}
\label{lem.nestcomp}
    If $\{\snm(\Theta^{i}(\myvs))\}_{i \in [\bar{h}]}$  is nested and complete, then  $\lambda^{1} \leq \cdots \leq \lambda^{\bar{h}} = \lambda^\star$.
\end{lemma}
\begin{proof}{Proof.}
    Larger inner approximations imply better relaxations and larger lower bounds. Since $ \nm(\myvs) \subseteq  \snm( \Theta^{\bar{h}}(\myvs))$, the lower bound $\lambda^{\bar{h}}$ is at least the optimal value $\lambda^\star$. It follows that $\lambda^{\bar{h}} = \lambda^\star$. This completes the proof.
\Halmos \end{proof}

We now need to determine the parameters $\Theta^{i}(\myvs),\fL^i(\myvs)$.
We illustrate how to apply the hierarchical partition and combine it  with the standard and Lovász extension methods to obtain these parameters. This gives us two different nested and complete families of cones of binary non-negative polynomials. Recall that we assume the NNS component $\nn(f)$ of $f$ is signed-supported within the vector $\myvs^{\snn}$ and its PS component $\ppn(f)$ is signed-supported within the vector $\myvs^{\spn}$. Thus, we can obtain two implementations of the template hierarchy of relaxations for lower bounding the binary polynomial $f$.

 \subsection{Standard signed relaxations.}

 We  present the combination of the  hierarchical partition with the standard extension, for which  the base set $C$ is chosen as the support $ \supp(\myvs^{\spn})$ of $\myvs^{\spn}$.  This relaxation method is oriented toward a hierarchical linearization of monomials in $\ppn(f)$.

 The support $ \supp(\myvs^{\spn})$ contains  exponents $\va$ of all monomials $\vx^\va$ in $\ppn(f)$. Consider the $i$-th layer $V^i(C)$ of the partition tree. Given the construction rule of the tree, this layer $V^i(C)$ is a partition of the exponents in $\supp(\myvs^{\spn})$. We construct a signed support vector  $\vth^{i,k}$ for each node $V^i_k(C)$ ($k \in [\ell^i]$) as follows. We decompose  $\vth^{i,k}$ as $\vth^{i,k,{\snn}} + \vth^{i,k,{\spn}}$. The signed support vector $\vth^{i,k,{\snn}}$ is the same as $\myvs^{\snn}$, since it is easy to certify the non-negativity of the corresponding NNS component. The signed support vector $\vth^{i,k,{\spn}}$ is signed-supported within $\myvs^{\spn}$, where entries $\vth^{i,k,{\spn}}_\va$  indexed by exponents   $\va \in V^i_k(C)$ are ones, and all other entries are zeros. We then construct an exact set $\cM(\vth^{i,k,{\spn}})$  of overestimation matrices w.r.t. $\vth^{i,k,{\spn}}$ using the standard extension. This allows for linearizing any PS polynomial signed-supported within $\vth^{i,k,{\spn}}$. In this way, we produce the parameters $ \Theta^{i}(\myvs), \fL^i(\myvs)$.

 \begin{theorem}
\label{prop.nestsedsparse}
The maximum depth, equivalently the number of layers, satisfies $ \bar{h} \le \ceil{\log{m_{{\spn}}}}+1 $, and $\{\snm(\Theta^{i}(\myvs))\}_{i \in [\bar{h}]}$ is nested and complete.
\end{theorem}
\begin{proof}{Proof.}
As $|C| = |\supp(\myvs^{\spn})| = m_{{\spn}}$,   it follows from \Cref{prop.arb}  that $\bar{h} \le \ceil{\log{m_{{\spn}}}}+1$. To prove that $\{\snm(\Theta^{i}(\myvs))\}_{i \in [\bar{h}]}$ is nested, it suffices to show that the set $ \nm(\vth^{i,k})$ associated with any node indexed at height $i$ and $k \in [\ell^i]$ includes the set $ \nm(\vth^{i-1,k'})$ associated with its subnode indexed at height $i-1$ and $k' \in [\ell^{i-1}]$, because this implies that any sum in $\snm(\Theta^{i-1}(\myvs))$ is also in $\snm(\Theta^{i}(\myvs))$. As $ \supp(\vth^{i-1,k',{\spn}}) \subseteq \supp(\vth^{i,k,{\spn}}) $ and $\vth^{i,k,{\snn}} = \vth^{i-1,k',{\snn}}$, the set $\nm(\vth^{i,k})$ includes more binary non-negative polynomials  than  $\nm(\vth^{i-1,k'})$. Since $\vth^{\bar{h},1,{\snn}} = \myvs^{\snn}$ and $\vth^{\bar{h},1,{\spn}} = \myvs^{\spn}$, $\{\snm(\Theta^{i}(\myvs))\}_{i \in [\bar{h}]}$ is also complete.
\Halmos \end{proof}

Implementing the template relaxation \eqref{eq.templatesparserelax} with the proposed  $\snm(\Theta^{i}(\myvs))$ yields the \emph{level-$i$ standard signed relaxation}. By \Cref{lem.nestcomp} and \Cref{prop.nestsedsparse}, we obtain a hierarchy of standard signed relaxations with increasing lower bounds that converge to the optimal value of the BPO.
As the number $m_{{\spn}}$ of  monomials indexed by exponents in $\supp(\myvs^{\spn})$ is at most $2^n$, the last level $\bar{h}$ is at most $n+1$, similar to the Lasserre hierarchy. In the sparse case, $\bar{h} \le \ceil{\log{m_{{\spn}}}}+1$ could be much smaller than $n$.

We analyze the encoding size of the cone $\snm(\Theta^i(\myvs))$, which naturally implies the complexity of linear optimization over this cone.
\begin{proposition}
For $ i \in [\ceil{\log{m_{{\spn}}}}+1]$, the
encoding size of $\snm(\Theta^i(\myvs))$ is  $\bO \left(m_{{\snn}} d_{{\snn}} m_{{\spn}} d_{{\spn}}^{2^i}  \right)$ (the parameters are as in \Cref{assm}).
\end{proposition}
\begin{proof}{Proof.}
The encoding size is the sum of the encoding sizes of its sub-cones, so it is
$\bO \left( m_{{\snn}} d_{{\snn}}  \sum_{k \in [\ell^i] }\Gamma(\vth^{i,k,{\spn}})\right)$  by  \Cref{thm.nn} and \Cref{thm.extend}.  Note that the encoding size of the PS cone does not contribute to the major complexity, we do some simplification next.
 By \Cref{prop.arb}, every node at height $i$ contains at most $2^{i-1}$ exponents. Using the standard extension, each such node satisfies
\[
\Gamma(\vth^{i,k,{\spn}})\le d_{\spn}^{\,2^{i-1}}\le d_{\spn}^{\,2^i}.
\]
Since $\ell^i=\bO(m_{\spn})$, we obtain
\[
\sum_{k\in[\ell^i]}\Gamma(\vth^{i,k,{\spn}})
=\bO(m_{\spn}d_{\spn}^{2^i}).
\]
\Halmos \end{proof}

This encoding-size bound implies that the level-$i$ standard signed relaxation is fixed-parameter tractable (FPT) by the ellipsoid algorithm, given a fixed level number $i$.

 \subsection{Lovász signed relaxations.}

 We  present the combination of the  hierarchical partition  with the relaxed Lovász extension, for which  the base set $C$ is chosen as the variable index set $\cN(\myvs^{\spn})$ (which indexes all variables appearing in monomials $\vx^\va$ for $\va \in \supp(\myvs^{\spn})$).  This relaxation method is oriented toward a hierarchical linearization of PS polynomials with variable indices in $\cN(\myvs^{\spn})$.

  Consider the $i$-th layer $V^i(C)$ of the partition tree. Given the construction rule of the tree, this layer $V^i(C)$ is a partition of the variable indices in $\cN(\myvs^{\spn})$. We construct a signed support vector  $\vth^{i,k}$ for each node $V^i_k(C)$ ($k \in [\ell^i]$) as follows. We decompose  $\vth^{i,k}$ as $\vth^{i,k,{\snn}} + \vth^{i,k,{\spn}}$. The signed support vector $\vth^{i,k,{\snn}}$ is the same as $\myvs^{\snn}$. The signed support vector $\vth^{i,k,{\spn}}$ is not necessarily signed-supported within $\myvs^{\spn}$.  We generate a  set $A^i_k(C) \deq\{\va \in \bB_{2:n}: \supp(\va) \subseteq  V^i_k(C)\}$ of exponents from $V^i_k(C)$. The entries $\vth^{i,k,{\spn}}_\va$  indexed by exponents   $\va \in A^i_k(C)$ are ones, and all other entries are zeros. We then construct $\cM(\vth^{i,k,{\spn}})$ as an exact set of overestimation matrices w.r.t. $\vth^{i,k,{\spn}}$ using the Lovász extension. This allows for linearizing any PS polynomial with variables indexed by $V^i_k(C)$. In this way, we produce the parameters $ \Theta^{i}(\myvs), \fL^i(\myvs)$.

\begin{theorem}
\label{prop.nestsetdense}

The maximum depth, equivalently the number of layers, satisfies $ \bar{h} \le \ceil{\log{n_{{\spn}}}}+1 $, and  $\{\snm(\Theta^{i}(\myvs))\}_{i \in [\bar{h}]}$ is nested and complete.
\end{theorem}
\begin{proof}{Proof.}
As $|C| = |\cN(\myvs^{\spn})| = n_{{\spn}}$,   it follows from \Cref{prop.arb}  that $\bar{h} \le \ceil{\log{n_{{\spn}}}}+1$. To prove that $\{\snm(\Theta^{i}(\myvs))\}_{i \in [\bar{h}]}$ is nested, it suffices to show that the set $ \nm(\vth^{i,k})$ associated with any node indexed at height $i$ and $k \in [\ell^i]$ includes the set $ \nm(\vth^{i-1,k'})$ associated with its subnode indexed at height $i-1$ and $k' \in [\ell^{i-1}]$, because this implies that any sum in $\snm(\Theta^{i-1}(\myvs))$ is also in $\snm(\Theta^{i}(\myvs))$. As $ \supp(\vth^{i-1,k',{\spn}}) \subseteq \supp(\vth^{i,k,{\spn}}) $ and $\vth^{i,k,{\snn}} = \vth^{i-1,k',{\snn}}$, the set $\nm(\vth^{i,k})$ includes more binary non-negative polynomials  than $\nm(\vth^{i-1,k'})$. At the root node, $V^{\bar h}_1(C)=\cN(\myvs^{\spn})$, and therefore
\[
\supp(\vth^{\bar{h},1,{\spn}})
=
\{\va\in\bB_{2:n}: \supp(\va)\subseteq \cN(\myvs^{\spn})\}.
\]
Thus $\myvs^{\spn} \preceq \vth^{\bar{h},1,{\spn}}$. Since $\vth^{\bar{h},1,{\snn}} = \myvs^{\snn}$, the family $\{\snm(\Theta^{i}(\myvs))\}_{i \in [\bar{h}]}$ is complete.
\Halmos \end{proof}

Implementing the template hierarchy of relaxations \eqref{eq.templatesparserelax} yields the \emph{level-$i$ Lovász signed relaxation}.
By \Cref{lem.nestcomp} and \Cref{prop.nestsetdense}, we obtain a hierarchy of Lovász  signed relaxations with increasing lower bounds that converge to the optimal value of the BPO by level $\bar{h}$. Note that the last level satisfies $\bar{h} \le \ceil{\log{n_{{\spn}}}}+1$. We analyze the encoding size  of the cone $\snm(\Theta^{i}(\myvs))$.
\begin{proposition}
For $ i \in [\ceil{\log{n_{{\spn}}}}+1]$, the
encoding size of $\snm(\Theta^{i}(\myvs))$ is  $\bO \left(m_{{\snn}} d_{{\snn}} m_{{\spn}} 2^{2^i} \right)$, where the parameters are defined as in \Cref{assm}.
\end{proposition}
\begin{proof}{Proof.}
The encoding size is the sum of the encoding sizes of its sub-cones $\nm(\vth^k)$, so it is
$\bO \left( m_{{\snn}} d_{{\snn}}  \sum_{k \in [\ell^i] }\Gamma(\vth^{i,k,{\spn}})\right)$ by  \Cref{thm.nn} and \Cref{thm.extend}. Note that the encoding size of the PS cone does not contribute to the major complexity, so we do some simplification next.
 By \Cref{prop.arb}, every  $|V^i_k(\cN(\myvs^{\spn}))| \le 2^i$. Using the Lovász extension, by \Cref{prop.filter}, every $\Gamma(\vth^{i,k,{\spn}}) \le 2^{2^{i}}$.  We note that  $\ell^i = \bO(n_{{\spn}}) \le  \bO(m_{{\spn}})$, so  $ \sum_{k \in [\ell^i] }\Gamma(\vth^{i,k,{\spn}})= \bO(m_{{\spn}} 2^{2^i} )$.  Then the result follows.
\Halmos \end{proof}

 This encoding size implies that the level-$i$ Lovász signed relaxation is FPT by the ellipsoid algorithm, given a fixed level number $i$.

The results show that, in the worst case, the number of levels in the Lovász signed relaxations is smaller than that in the standard signed relaxations. In practice, one usually  uses relaxations at lower levels. This means that the signed certificates $f^k$ in \eqref{eq.snm} should have small encoding sizes. Regarding the standard signed relaxation,  the PS component  of $f^k$ matches  monomials of $f$ indexed by $ \supp(\vth^{i,k,{\spn}})$, which is a small subset of $\supp(\myvs^{\spn})$. For the Lovász signed relaxation, the PS component of $f^k$ corresponds to the restriction of $f$ to the variables indexed by $V^{i}_k(\cN(\myvs^{\spn}))$; these variables form a small subset of $\cN(\myvs^{\spn})$. This is a noteworthy  difference.

\section{Computational results.}
\label{sec.cresult}
In this section, we present the results of our computational experiments for the proposed relaxations. The source code, data, examples, and detailed results can be found in our online repository: \href{https://github.com/lidingxu/BPOrelaxations}{github.com/lidingxu/BPOrelaxations}. The experiments were conducted on a server with an Intel Xeon W-2245 CPU @ 3.90GHz and 128GB main memory. The reported runs used Julia 1.8.3, JuMP 1.13.0, MosekTools 0.15.0, and MOSEK 10.1.2. Unless stated otherwise, runs used one thread, default MOSEK tolerances, and a one-hour time limit per relaxation.

The Sherali-Adams, Lasserre, and signed relaxations were tested on \maxcut problems.
  Consider an undirected graph $G=(V,E,w)$, where $V$ is the set of nodes, $E$ is the set of edges, and $w$ is a weight function over subsets of $E$. For a subset $S$ of $V$, its associated cut capacity is  the sum of the weights of edges with one end node in $S$ and the other end node in $V \setminus S$. The \maxcut problem aims at finding a subset $S \subseteq V$ with the  maximum cut capacity. Let $V = [n]$, and we use a binary variable vector $\vx \in \bB$ indicating whether vertices belong to $S$. The problem can be formulated as a quadratic BPO problem: $
    \max_{\vx \in \bB}  \sum_{\{i, j\} \in E}w_{ij} ((1-x_i)x_j + x_i(1-x_j)).$

The Biq Mac library by \citet{wiegele2007biq} offers a collection of medium-size \maxcut instances.
 Our benchmark consists of four sub-benchmarks:
 \begin{itemize}
     \item 20 ``pm1s\_n.i'' instances: graphs with edge weights chosen uniformly from $\{-1,0,1\}$, density 0.1, and $n=80,100$, generated with the \texttt{rudy} code by \citet{rendl2010solving,rinaldi1998rudy};
     \item 10 ``w01\_100.'' instances:  graphs with integer edge weights chosen from $[-10,10]$, density $0.1$, and $n=100$, generated with the \texttt{rudy} code;
     \item 9 ``t2gn\_seed'' instances: toroidal grid graphs with Gaussian-distributed weights and dimensions $n =100,225,400$, generated by \citet{liers2004computing} from applications in statistical physics;
     \item 9 ``t3gn\_seed'' instances: toroidal grid graphs with Gaussian-distributed weights and dimensions $n =125,216,343$, generated by \citet{liers2004computing} from applications in statistical physics.
 \end{itemize}

The relaxations were tested in the following settings:

 \begin{itemize}
     \item \texttt{Sherali-Adams 1}: the first level Sherali-Adams relaxation, where the non-negative certificate is $\sum_{j,j' \in \cN}f_{j j'}x_j (1-x_{j'}) + \sum_{j\in \cN}f_{j} x_j  + \sum_{j\in \cN}\bar{f}_{j} (1-x_{j})$ with $f_{jj'} \ge 0, f_j \ge 0, \bar{f}_{j} \ge 0$.
     \item \texttt{Lasserre 1}: the first level Lasserre relaxation, where the non-negative certificate is $\sum_{i \in I}(q^i(\vx))^2$, with $I$ denoting the finite index set of squared linear forms and each $q^i(\vx)$ a linear polynomial.
     \item \texttt{Standard signed $i$}: the $i$-th level standard signed relaxation for $i = 1, 2, 3$.
 \end{itemize}

We anticipated that the low-level Lovász signed relaxations would exhibit similarity to their corresponding counterparts in the low-level standard signed relaxations, since the PS component in each signed certificate only contains a few monomials in these relaxations. For simplicity of implementation, our experimentation focused solely on the standard signed relaxations. In our experiment, we captured two key metrics: the time taken to solve a relaxation and the relative duality gap associated with that relaxation. The relative duality gap is defined through the following formula: given $\lambda'$ as the optimal value of a relaxation and $\lambda^\ast$ as the optimal value of the BPO obtained from the Biq Mac library, the relative duality gap is calculated as $(\lambda'-\lambda^\ast)/\lambda'$.
  In \Cref{tb.perf}, we reported the shift geometric means of the performance metrics of the relaxations on each sub-benchmark and overall benchmark, and we used boldface to mark the ``winner'' setting for each metric.

\begin{table} [htbp]
\footnotesize
\centering
\resizebox{0.85\columnwidth}{!}{
\begin{tabular}{c|*{2}{c}|*{2}{c}|*{2}{c}|*{2}{c}|*{2}{c}}
\toprule
\multirow{2}{*}{Setting} &
\multicolumn{2}{c|}{pm1s\_ni} &
\multicolumn{2}{c|}{w01\_100} &
\multicolumn{2}{c|}{t2gn\_seed} &
\multicolumn{2}{c|}{t3gn\_seed} &
\multicolumn{2}{c}{All}  \\
&
{gap} &
{time}     &
{gap} &
{time}     &
{gap} &
{time}     &
{gap} &
{time}     &
{gap} &
{time}      \\
\midrule
Sherali-Adams 1  & 0.509 & \textbf{0.0} & 0.47 & \textbf{0.0} & 0.173 & \textbf{0.0} & 0.277 & \textbf{0.0} & 0.373 & \textbf{0.0}   \\
Lasserre 1  & \textbf{0.127} & 4.1 & \textbf{0.115} & 7.4 & 0.183 & 340.8 & 0.189 & 443.5 & \textbf{0.144} & 27.9   \\
Standard signed 1  & 0.275 & 7.6 & 0.252 & 12.2 & 0.104 & 9.3 & 0.167 & 24.6 & 0.21 & 11.0   \\
Standard signed 2  & 0.253 & 14.7 & 0.24 & 26.4 & 0.095 & 55.5 & 0.161 & 125.0 & 0.196 & 32.1   \\
Standard signed 3  & 0.239 & 29.3 & 0.229 & 49.2 & \textbf{0.088} & 128.7 & \textbf{0.156} & 304.9 & 0.186 & 67.2  \\
\bottomrule
\end{tabular}
}
\caption{Summary of performance metrics.}\label{tb.perf}
\end{table}

Our initial observation reveals that the Sherali-Adams 1 relaxation exhibits the weakest performance in terms of closing the duality gap; however, it is the fastest to compute, with almost negligible computation time. Among the small instances (pm1s\_ni, w01\_100, t3gn\_seed), the standard signed relaxations require more  computation time compared to the Lasserre 1 relaxation, yet their computation time is generally comparable. Notably, the bounds of the standard signed relaxations are slightly less favorable than those of the Lasserre 1 relaxation in these cases.

In contrast, for the larger instances (t2gn\_seed and t3gn\_seed), the standard signed relaxations exhibit notably shorter computation times compared to the Lasserre 1 relaxation. Furthermore, the duality gaps associated with the standard signed relaxations are substantially smaller in comparison to those of the Lasserre 1 relaxation. This divergence in performance is attributable to the fact that, in the case of four large instances, the computation of the Lasserre 1 relaxation surpasses the one-hour time threshold, leading to relative duality gaps being registered as 1.

\begin{figure}[h]
    \centering
    \includegraphics[width=0.9\textwidth]{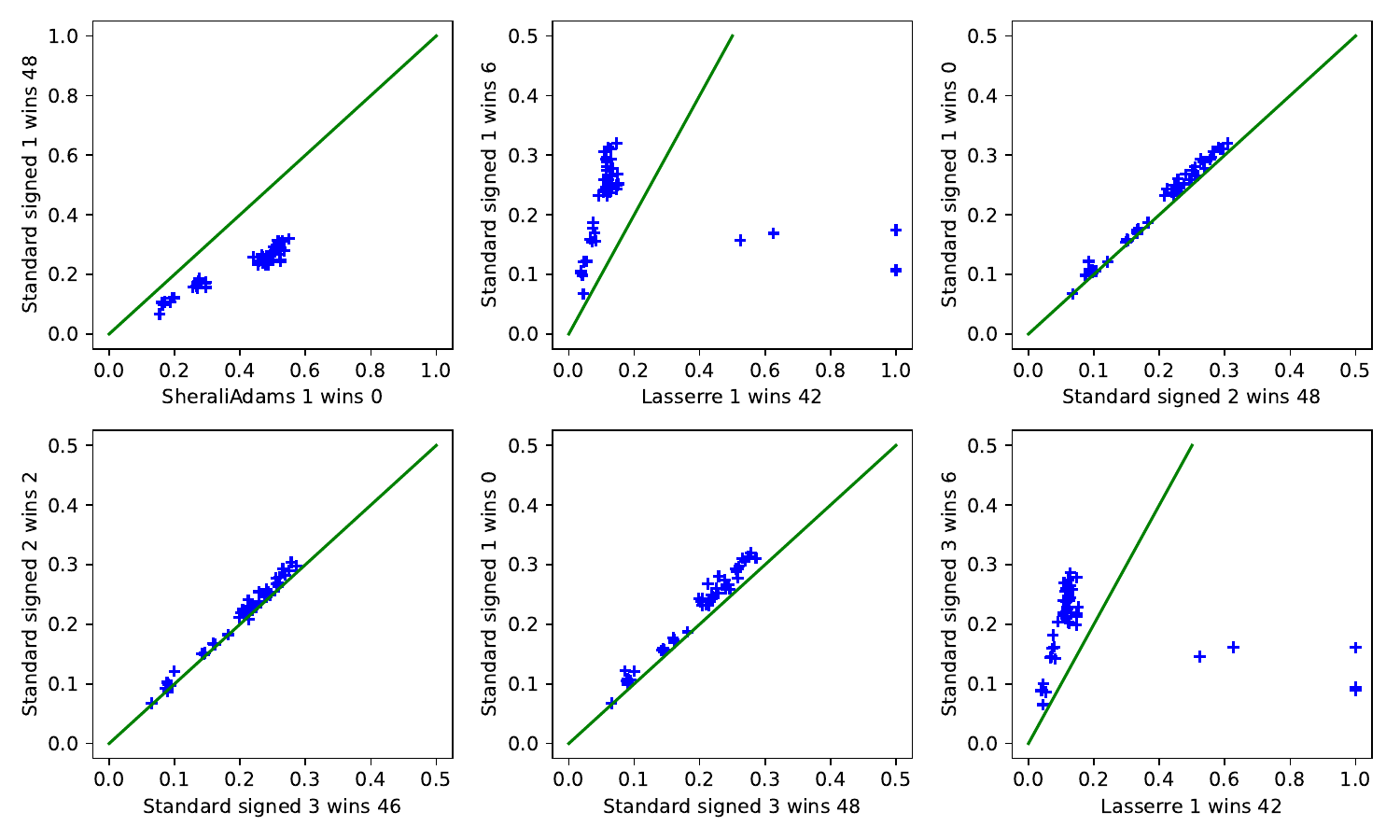}
    \caption{Instance-wise duality gap between pairs of relaxations}
    \label{fig:bpo}
\end{figure}

The scatter plots in \Cref{fig:bpo} serve to contrast the relative duality gaps of different relaxations across each instance. Each scatter plot also  reports the count of instances where one relaxation surpasses the other, denoted as ``win'' instances. We observe that, as the level number $i$ increases, the $i$-th standard signed relaxation consistently yields better bounds. Notably, in the final scatter plot comparing the Lasserre 1 relaxation with the standard signed relaxation 3, there are four outliers. These outliers have a detrimental impact on the performance of the Lasserre 1 relaxation.

From the analysis, we can deduce that the lower-level standard signed relaxations generally yield somewhat less favorable bounds than those of the initial-level Lasserre relaxation. However, unlike the initial Sherali-Adams relaxations, the bounds are still reasonably strong. Notably, the scalability of the standard signed relaxations surpasses that of the first-level Lasserre relaxation when dealing with larger problem instances.

\section{Conclusion.}
\label{sec.conc}

We characterize binary non-negativity for binary polynomials based on their signed support vectors. We identify a conic set of binary polynomials whose binary non-negativity can be certified in strongly polynomial time. We construct an LP representation of such binary polynomials for a given signed support pattern. This yields  new binary non-negativity certificates, called signed certificates. We present an LP to search for the refined signed support decomposition of any binary polynomial as signed certificates.  Using hierarchical partitions, we present two general approaches to construct nested families of cones of binary non-negative polynomials. These families yield new hierarchies of relaxations for BPO. Future work could focus on designing efficient numerical algorithms for solving the relaxations.
%
%
%
%

\appendix
\section{Details on min-cut-based separation.}
\label{sec.append}
Following \Cref{subsec.sepalgo}, we provide details here on the \mincut-based separation algorithm and the proof of \Cref{thm.mincut}.

    At this point, we consider some preprocessing steps of the \mincut problem in $G^c$. First, due to the coefficients of $f^c$, the capacities of $G^c$ are non-negative.  Second, let $A_j\deq \{\va \in A: j \in \supp(\va)\}$ be the set of exponent vectors whose supports contain variable index $j$, \ie there are edges from $\va \in A_j$ to $j$. Since the network is directed, its edges cannot originate from $\cN$ and connect to nodes in $A$. As a result, a cut has finite capacity, i.e., it crosses no unlimited-capacity edge $(\va,j)$, if and only if, for all $j \in \cN$ and $\va \in A_j$, the condition $u_j \geq u_\va$ is satisfied. In other words, $u_\va = 1$ and $u_j = 0$ cannot both occur. Third, for any $\va \in A$ we assume that, if $u_\va = 0$ and for all $j \in \supp(\va)$, $u_j = 1$, then the capacity $C^c(\vvu)$ is given by the sum of capacities over the edges $(\mysl, \va)$ and $(j,\sr)$. As we are only interested in \mincut,  we can set $u_\va = 1$  and thus remove the capacity from $(\mysl, \va)$. We can consider $\vvu$ for which $u_\va = \prod_{j \in \supp(\va)} u_j$ ($\va \in A$). In this way, we can preprocess the cut space by restricting it to $U_f  \deq  \{ \vvu \in U: \forall j \in \cN_f, u_j = 1 \land \forall \va \in A, u_\va = \prod_{j \in \supp(\va)} u_j\}$.

Next, we discuss the relationship between the  preprocessed \mincut problem $\min\limits_{\vvu \in U_f} C^c(\vvu)$ in $G^c$ and the binary polynomial minimization problem $\min\limits_{\vx \in \bB_f} f^c(\vx)$.
We define a map from  binary vectors $\vx \in \bB$ to binary cut labels $\vvu \in U$:
\begin{equation}
\label{eq.map}
    \zeta: \bB \to \{0,1\}^V, \vx \mapsto \vvu, \textup{ s.t. } \vvu_{\vv} = \begin{cases}
       x_j & \mbox{ where }\vv = j \in \cN \\
        \vx^\va & \mbox{ where }\vv = \va  \in A \\
       1 & \mbox{ where }\vv = \mysl \\
       0 & \mbox{ where }\vv = \sr
    \end{cases}\; \qquad \forall \vv \in V,
\end{equation}
and show that the two preprocessed problems are equivalent under the action of the map $\zeta$.

\begin{lemma}
\label{thm.f2cut}
   $\zeta$ is a bijective map from $\bB_f$ to $U_f$, and for any $\vx \in \bB_f$, $f^c(\vx)= C^c(\zeta(\vx))$. Moreover, $\min\limits_{\vx \in \bB_f}f^c(\vx) = \min\limits_{\vvu \in U_f}C^c(\vvu) = \min\limits_{\vvu \in U}C^c(\vvu)$, and any solution $\vvu'$ to $ \min\limits_{\vvu \in U}C^c(\vvu)$ can be converted into a solution to $\min_{\vvu\in U_f}C^c(\vvu)$ without increasing cut capacity by setting $u_j=1$ for all $j\in\cN_f$ and then setting $u_\va=\prod_{j\in\supp(\va)}u_j$ for all $\va\in A$.
\end{lemma}
  \begin{proof}{Proof.}
   We prove $\zeta$ is bijective.
   Note that the entries of $\zeta(\vx)$ indexed by $\cN$ are $\vx$.
 We first prove that $\zeta(\bB_f) \subseteq U_f$. Given $\vx \in \bB_f$, for all $j \in \cN_f$, $x_j = \zeta(\vx)_j =1$. Moreover, for all $\va \in A$,  $\zeta(\vx)_\va =\prod_{j \in \supp(\va)} \zeta(\vx)_j =  \prod_{j \in \supp(\va)} x_j  = \vx^\va  $. Thus, $ \zeta(\vx) \in U_f$.
We next prove that $\zeta$ is injective. For any $\vx \ne \vx' \in \bB_f$, we have that $(\zeta(\vx)_j)_{j \in \cN} = \vx \ne \vx' = (\zeta(\vx')_j)_{j \in \cN}$, which implies that $\zeta(\vx) \ne  \zeta(\vx')$. We then prove that $\zeta$ is surjective. For any $\vvu \in U_f$,  construct $\vx$ for which $ x_j = u_j$ ($j \in \cN$). Since, for all $j \in \cN_f$, $x_j = u_j = 1$, and  for all $\va \in A$,  $u_\va =\prod_{j \in \supp(\va)} u_j =  \prod_{j \in \supp(\va)} x_j  = \vx^\va  $, it follows that $\vx \in \bB_f$ and $\vvu = \zeta(\vx)$. Finally, given any $\vx \in \bB_f$, we show that $f^c(\vx)= C^c(\zeta(\vx))$. Let $\vvu =\zeta(\vx)$. Recall that $f^c(\vx)=\sum_{\va \in A} -f_{\va} (1-\vx^{\va}) +  \sum_{j \in \cN} f^+_j x_j$. If $x_j = 1$ ($f^c(\vx)$ contains $f^+_j$), then $u_j = 1$ ($(j, \sr)$ crosses the cut); otherwise, $u_j = 0$ ($(j, \sr)$ does not cross the cut). If $\vx^\va = 0$ ($f^c(\vx)$ contains $-f_{\va}$), then $u_\va = 0$ ($(\mysl,\va)$ crosses  the cut); otherwise, $u_\va = 1$ ($(\mysl,\va)$ does not cross the cut). We find that $f^c(\vx) = C^c(\vvu)$. As a result, we have that $\min\limits_{\vx \in \bB_f}f^c(\vx) = \min\limits_{\vvu \in U_f}C^c(\vvu)$. For $j\in\cN_f$, the edge $(j,\sr)$ has capacity $f_j^+=0$, and setting $u_j=1$ can only remove crossings of unlimited-capacity edges $(\va,j)$. Afterward, setting $u_\va=\prod_{j\in\supp(\va)}u_j$ cannot increase capacity: if all adjacent $u_j$ are one, changing $u_\va$ from zero to one removes the source edge $(\mysl,\va)$ from the cut; otherwise finite capacity already forces $u_\va=0$.
\Halmos
 \end{proof}

Therefore, we can solve the non-restricted \mincut problem $\min\limits_{\vvu \in U} C^c(\vvu)$ (which is the standard  form accepted by \mincut algorithms) and recover a solution to $\min\limits_{\vvu \in U_f} C^c(\vvu)$. We now obtain the full reduction of the separation problem. The problem of solving $\min\limits_{\vx \in \bB}f(\vx)$ is reduced to finding a solution to the modified problem $\min\limits_{\vx \in \bB_f}f^c(\vx)$, which is again reduced to a \mincut problem $\min\limits_{\vvu \in U_f}C^c(\vvu)$, with the standard form $\min\limits_{\vvu \in U}C^c(\vvu)$. Finally, \Cref{algo.sep} outlines
 the separation procedure.

\begin{corollary}
\label{thm.sepsub}
For any $\myvs\in\{-1,0,1\}^{\bB}$ with $\myvs_{\bB_{2:n}}\le \vz$ and $\bB_{0:1}\subseteq\supp(\myvs)$, \Cref{algo.sep} solves the separation problem for $\nnn(\myvs)$ in $\bO(m^2d)$ time.
\end{corollary}
\begin{proof}{Proof.}
 The capacities of $G^c$ are coefficients of $f$, which are assumed to be fixed-size rationals.
 Since $\bB_{0:1}\subseteq\supp(\myvs)$, the support pattern contains the constant monomial and all $n$ linear monomials, so $n\le m-1=\bO(m)$. Also, $|A|\le m$ and $\sum_{\va\in A}|\supp(\va)|\le |A|d=\bO(md)$. Hence $G^c$ has $|A|+n+2=\bO(m)$ nodes and $|A|+\sum_{\va\in A}|\supp(\va)|+n=\bO(md)$ edges. We use the \maxflow algorithm by \citet{orlin2013max}, which runs in time proportional to the product of the numbers of nodes and edges, giving $\bO(m^2d)$ time.
\Halmos \end{proof}

  \begin{algorithm}[htbp]
\SetAlgoLined
 \textbf{Input:} $f$ given as in \eqref{eq.sep.f}\;
 \textbf{Output:} a vector $\vx^\ast\in\bB$ with $f(\vx^\ast)<0$, if one exists; otherwise certify that $\min_{\vx\in\bB}f(\vx)\ge 0$\;
Extract $f^c$ from $f$ as in \eqref{eq.sep.fc} \;
Construct the flow network $G^c$ from $f^c$\;
Find the min $\mysl\sr$-cut $\vvu^\ast$\;
Construct $\vx^\ast$ by setting $x^\ast_j=u^\ast_j$ for all $j\in\cN$, and then set $x^\ast_j=1$ for all $j\in\cN_f$\;
\If{$f(\vx^\ast) < 0$}{
 Output $\vx^\ast$\;
}
\Else{
 Report that no violated non-negativity sub-constraint exists\;
}
\caption{\mincut based separation algorithm}
\label{algo.sep}
\end{algorithm}

Notably, \Cref{thm.f2cut} and \Cref{thm.sepsub} complete the proof for \Cref{thm.mincut}.
 The following example illustrates the algorithm and its reformulation procedures (codes for verification are provided in \Cref{sec.cresult}).

\begin{example}
\label{exampl.mcf}
Consider $f(\vx) =- x_2 x_3 -2 x_1 x_3 x_4 -5 x_3 x_5 + x_2 + x_3-x_4 + 7$ as the NNS component of the binary polynomial in \Cref{exampl.pre}. Its minimum is attained at  $\vx^\ast = (1,0,1,1,1)$. We have that $f^a = -1, f^b(\vx) = (1-x_2 x_3) +2 (1-x_1 x_3 x_4) +5 (1-x_3 x_5) + x_2 + x_3-x_4$ with $f^b(\vx^\ast) = 1$, and $f^c(\vx) = (1-x_2 x_3) +2 (1-x_1 x_3 x_4) +5 (1-x_3 x_5) + x_2 + x_3$  with $f^c(\vx^\ast) = 2$. The graph $G^c$ (depicted in \Cref{fig.flow2}) has left node set $A = \{\va_1 = \veu_2 + \veu_3, \va_2 = \veu_1 + \veu_3 +\veu_4, \va_3 = \veu_3 + \veu_5\}$ and right node set $\cN = \{1,\dots,5\}$. The capacities of $(\mysl,\va_1), (\mysl, \va_2), (\mysl,\va_3)$ are 1,2,5, respectively.  The capacities of $(1,\sr), (2,\sr), (3,\sr), (4,\sr), (5,\sr)$ are 0,1,1,0,0, respectively. The min $\mysl\sr$-cut has label $\vvu^\ast$ with $u^\ast_{\va_1} = 0, u^\ast_{\va_2} = 1, u^\ast_{\va_3}= 1, u^\ast_1 = 1, u^\ast_2 = 0, u^\ast_3 = 1, u^\ast_4 = 1, u^\ast_5 = 1$. Thus, the cut crosses the edges $(\mysl,\va_1), (1,\sr), (3,\sr), (4,\sr), (5,\sr)$, so its capacity is $C^c(\vvu^\ast) = 1 + 0 + 1 + 0 + 0= 2 = f^c(\vx^\ast)$. Note that $(u^\ast_1,\dots, u^\ast_5) = \vx^\ast$ and $\zeta(\vx^\ast) = \vvu^\ast$ ($u^\ast_{\va_1} = 0 = x^\ast_2 x^\ast_3, u^\ast_{\va_2} = 1 =  x^\ast_1 x^\ast_3 x^\ast_4,  u^\ast_{\va_3} = 1 =  x^\ast_3 x^\ast_5$). As $f(\vx^\ast) = 0$ is the minimum value, we can conclude that $f$ is binary non-negative.
\end{example}

\begin{figure}[H]
    \centering
\begin{tikzpicture}[thick,
  fsnode/.style={fill=myblue,draw,rectangle},
  ssnode/.style={fill=mygreen,,draw,rectangle},
  every fit/.style={draw,inner sep=2pt,text width=2cm},
  ->,shorten >= 2pt,shorten <= 2pt
]

\begin{scope}[yshift=-0.5cm, start chain=going below,node distance=9mm]
\node[fsnode,on chain] (f1) [] { $u^\ast_{\va_1} = 0$};
\node[fsnode,on chain] (f2) [] { $u^\ast_{\va_2} = 1$};
\node[fsnode,on chain] (f3) [] { $u^\ast_{\va_3} = 1$};
\end{scope}

\begin{scope}[xshift=6cm, yshift=0cm, start chain=going below,node distance=4.5mm]
\node[ssnode,on chain] (s1) [] {$u^\ast_{1} = 1$};
\node[ssnode,on chain] (s2) [] {$u^\ast_{2} = 0$};
\node[ssnode,on chain] (s3) [] { $u^\ast_{3} = 1$};
\node[ssnode,on chain] (s4) [] { $u^\ast_{4} = 1$};
\node[ssnode,on chain] (s5) []  {$u^\ast_{5} = 1$};
\end{scope}

\node [fsnode, left=2cm of f2] (vs) {$u^\ast_{\mysl} =1$};
\node [ssnode, right=2.8cm of s3] (vt) {$u^\ast_{\sr} = 0$};
\node [myblue,fit=(f1) (f3),label=above:$A$] {};
\node [mygreen,fit=(s1) (s5),label=above: $\cN$] {};

\draw[dashed] (f1) -- node[above,pos=0.32]{$0/\infty$}(s2);
\draw[dashed] (f1) -- node[above,pos=0.32]{$1/\infty$}(s3);
\draw[dashed] (f2) -- node[above,pos=0.68]{$0/\infty$}(s1);
\draw[dashed] (f2) -- node[above,pos=0.68]{$0/\infty$}(s3);
\draw[dashed] (f2) -- node[above,pos=0.68]{$0/\infty$}(s4);
\draw[dashed] (f3) -- node[above,pos=0.3]{$0/\infty$}(s3);
\draw[dashed] (f3) -- node[above,pos=0.3]{$0/\infty$}(s5);
\draw[red] (vs) -- node[above]{$1/1$}(f1);
\draw[dashed] (vs) -- node[above]{$0/2$}(f2);
\draw[dashed] (vs) -- node[above]{$0/5$}(f3);
\draw[red] (s1) -- node[above,pos=0.4]{$0/0$}(vt);
\draw[dashed] (s2) -- node[above,pos=0.4]{$0/1$}(vt);
\draw[red] (s3) -- node[above,pos=0.4]{$1/1$}(vt);
\draw[red] (s4) -- node[above,pos=0.4]{$0/0$}(vt);
\draw[red] (s5) -- node[above,pos=0.4]{$0/0$}(vt);
\end{tikzpicture}
\caption{The network $G^c$ derived from $f^c$ in \Cref{exampl.mcf}. Nodes are labeled by their values in the \mincut. Edge labels are shown as flow/capacity, and the solid red edges are the edges crossing the cut.}
\label{fig.flow2}
\end{figure}

\section*{Acknowledgments.}
Most of the work done by the first author on this paper occurred while he was a Ph.D.~student at LIX CNRS, \'Ecole Polytechnique, Institut Polytechnique de Paris (IPP), sponsored by a grant of the IPP doctoral school, for which both authors are grateful.


\clearpage


\bibliographystyle{plainnat}

\bibliography{reference}

@misc{wiegele2007biq,
  title={{Biq Mac Library—A collection of Max-Cut and quadratic 0-1 programming instances of medium size}},
  author={Wiegele, Angelika},
  howpublished = {\url{https://biqmac.aau.at/biqmaclib.html}},
  year={2007}
}

@misc{rinaldi1998rudy,
  title={Rudy},
  author={Rinaldi, Giovanni},
  howpublished = {\url{http://www-user.tu-chemnitz.de/~helmberg/rudy.tar.gz}},
  year={1998}
}

@incollection{liers2004computing,
 Author = {Liers, Frauke and J{\"u}nger, Michael and Reinelt, Gerhard and Rinaldi, Giovanni},
 Title = {Computing exact ground states of hard {Ising} spin glass problems by branch-and-cut},
 BookTitle = {New optimization algorithms in physics.},
 Pages = {47--69},
 Year = {2004},
 Publisher = {Weinheim: Wiley-VCH},
 Editor={Alexander K. Hartmann and Heiko Riege}
}

@article{rendl2010solving,
 Author = {Rendl, Franz and Rinaldi, Giovanni and Wiegele, Angelika},
 Title = {Solving {Max}-cut to optimality by intersecting semidefinite and polyhedral relaxations},
 Journal = {Mathematical Programming},
 Volume = {121},
 Number = {2 (A)},
 Pages = {307--335},
 Year = {2010},
}

@article{helmberg1998solving,
 Author = {Helmberg, Christoph and Rendl, Franz},
 Title = {Solving quadratic (0,1)-problems by semidefinite programs and cutting planes},
 Journal = {Mathematical Programming},
 Volume = {82},
 Number = {3 (A)},
 Pages = {291--315},
 Year = {1998}
}

@article{krislock2014improved,
 Author = {Krislock, Nathan and Malick, J{\'e}r{\^o}me and Roupin, Fr{\'e}d{\'e}ric},
 Title = {Improved semidefinite bounding procedure for solving max-cut problems to optimality},
 Journal = {Mathematical Programming},
 Volume = {143},
 Number = {1-2 (A)},
 Pages = {61--86},
 Year = {2014}
}

@article{malick2013bridge,
  title={On the bridge between combinatorial optimization and nonlinear optimization: a family of semidefinite bounds for 0--1 quadratic problems leading to quasi-Newton methods},
  author={Malick, J{\'e}r{\^o}me and Roupin, Fr{\'e}d{\'e}ric},
  journal={Mathematical Programming},
  volume={140},
  number={1},
  pages={99--124},
  year={2013},
  publisher={Springer}
}

@book{fujishige2005submodular,
 Author = {Fujishige, Satoru},
 Title = {Submodular functions and optimization},
 Volume = {58},
 Year = {2005},
Series = {Annals of Discrete Mathematics},
 Publisher = {Amsterdam: Elsevier}
}

@article{nemhauser1978analysis,
  title={An analysis of approximations for maximizing submodular set functions—{I}},
  author={Nemhauser, George L and Wolsey, Laurence A and Fisher, Marshall L},
  journal={Mathematical Programming},
  volume={14},
  number={1},
  pages={265--294},
  year={1978},
  publisher={Springer}
}

@article{crama1993concave,
  title={Concave extensions for nonlinear 0--1 maximization problems},
  author={Crama, Yves},
  journal={Mathematical Programming},
  Number = {1 (A)},
  volume={61},
  pages={53--60},
  year={1993},
  publisher={Springer}
}

@article{grotschel1981ellipsoid,
  title={The ellipsoid method and its consequences in combinatorial optimization},
  author={Gr{\"o}tschel, Martin and Lov{\'a}sz, L{\'a}szl{\'o} and Schrijver, Alexander},
  journal={Combinatorica},
  volume={1},
  number={2},
  pages={169--197},
  year={1981},
  publisher={Springer}
}

@article{laurent2003comparison,
  title={{A comparison of the Sherali-Adams, Lov{\'a}sz-Schrijver, and Lasserre relaxations for 0--1 programming}},
  author={Laurent, Monique},
  journal={Mathematics of Operations Research},
  volume={28},
  number={3},
  pages={470--496},
  year={2003},
  publisher={INFORMS}
}

@article{gatermann2004symmetry,
title = {Symmetry groups, semidefinite programs, and sums of squares},
journal = {Journal of Pure and Applied Algebra},
volume = {192},
number = {1},
pages = {95-128},
year = {2004},
author = {Karin Gatermann and Pablo A. Parrilo},
  publisher={Elsevier}
}

@inproceedings{wang2019new,
author = {Wang, Jie and Li, Haokun and Xia, Bican},
title = {A New Sparse SOS Decomposition Algorithm Based on Term Sparsity},
year = {2019},
publisher = {ACM},
address = {New York, NY},
booktitle = {Proceedings of  ISSAC 2019},
pages = {347–354},
numpages = {8},
location = {Beijing, China}
}

@article{del2023polynomial,
  title={A polynomial-size extended formulation for the multilinear polytope of beta-acyclic hypergraphs},
  author={Del Pia, Alberto and Khajavirad, Aida},
  journal={Mathematical Programming},
  pages={1--33},
  year={2023},
  publisher={Springer}
}

@article{del2023complexity,
  title={On the complexity of binary polynomial optimization over acyclic hypergraphs},
  author={Del Pia, Alberto and Di Gregorio, Silvia},
  journal={Algorithmica},
  volume={85},
  number={8},
  pages={2189--2213},
  year={2023},
  publisher={Springer}
}

@article{del2017polyhedral,
  title={A polyhedral study of binary polynomial programs},
  author={Del Pia, Alberto and Khajavirad, Aida},
  journal={Mathematics of Operations Research},
  volume={42},
  number={2},
  pages={389--410},
  year={2017},
  publisher={INFORMS}
}

@article{wang2021tssos,
  title={{TSSOS: A moment-SOS hierarchy that exploits term sparsity}},
  author={Wang, Jie and Magron, Victor and Lasserre, Jean-Bernard},
  journal={SIAM Journal on Optimization},
  volume={31},
  number={1},
  pages={30--58},
  year={2021},
  publisher={SIAM}
}

@article{wang2020chordal,
 Author = {Wang, Jie and Magron, Victor and Lasserre, Jean-Bernard},
 Title = {Chordal-{TSSOS}: a moment-{SOS} hierarchy that exploits term sparsity with chordal extension},
 Journal = {SIAM Journal on Optimization},
 Volume = {31},
 Number = {1},
 Pages = {114--141},
 Year = {2021}
}

@article{murray2021newton,
  title={Newton polytopes and relative entropy optimization},
  author={Murray, Riley and Chandrasekaran, Venkat and Wierman, Adam},
  journal={Foundations of Computational Mathematics},
 Volume = {21},
 Number = {6},
  pages={1703--1737},
  year={2021},
  publisher={Springer}
}

@article{chandrasekaran2016relative,
  title={Relative entropy relaxations for signomial optimization},
  author={Chandrasekaran, Venkat and Shah, Parikshit},
  journal={SIAM Journal on Optimization},
  volume={26},
  number={2},
  pages={1147--1173},
  year={2016},
  publisher={SIAM}
}

@article{dressler2019approach,
  title={An approach to constrained polynomial optimization via nonnegative circuit polynomials and geometric programming},
  author={Dressler, Mareike and Iliman, Sadik and De Wolff, Timo},
  journal={Journal of Symbolic Computation},
  volume={91},
  pages={149--172},
  year={2019},
  publisher={Elsevier}
}

@article{papp2023duality,
  title={Duality of sum of nonnegative circuit polynomials and optimal {SONC} bounds},
  author={Papp, D{\'a}vid},
  journal={Journal of Symbolic Computation},
  volume={114},
  pages={246--266},
  year={2023},
  publisher={Elsevier}
}

@article{magron2023sonc,
  title={{SONC} optimization and exact nonnegativity certificates via second-order cone programming},
  author={Magron, Victor and Wang, Jie},
  journal={Journal of Symbolic Computation},
  volume={115},
  pages={346--370},
  year={2023},
  publisher={Elsevier}
}

@article{lasserre2002semidefinite,
  title={Semidefinite programming vs. {LP} relaxations for polynomial programming},
  author={Lasserre, Jean B},
  journal={Mathematics of Operations Research},
  volume={27},
  number={2},
  pages={347--360},
  year={2002},
  publisher={INFORMS}
}

@article{sherali1990hierarchy,
  title={A hierarchy of relaxations between the continuous and convex hull representations for zero-one programming problems},
  author={Sherali, Hanif D and Adams, Warren P},
  journal={SIAM Journal on Discrete Mathematics},
  volume={3},
  number={3},
  pages={411--430},
  year={1990},
  publisher={SIAM}
}

@article{sherali1992global,

 Author = {Sherali, Hanif D. and Tuncbilek, Cihan H.},
 Title = {A global optimization algorithm for polynomial programming problems using a reformulation-linearization technique},
 Journal = {Journal of Global Optimization},
 Volume = {2},
 Number = {1},
 Pages = {101--112},
 Year = {1992},
  publisher={Springer}
}

@article{sherali1997new, 
  title={New reformulation linearization/convexification relaxations for univariate and multivariate polynomial programming problems},
  author={Sherali, Hanif D and Tuncbilek, Cihan H},
  journal={Operations Research Letters},
  volume={21},
  number={1},
  pages={1--9},
  year={1997},
  publisher={Elsevier}
}

@article{sherali2012reduced,
  title={Reduced RLT representations for nonconvex polynomial programming problems},
  author={Sherali, Hanif D and Dalkiran, Evrim and Liberti, Leo},
  journal={Journal of Global Optimization},
  volume={52},
  pages={447--469},
  year={2012},
  publisher={Springer}
}

@article{tawarmalani2013explicit,
  title={Explicit convex and concave envelopes through polyhedral subdivisions},
  author={Tawarmalani, Mohit and Richard, Jean-Philippe P and Xiong, Chuanhui},
  journal={Mathematical Programming},
  volume={138},
  number={1},
  pages={531--577},
  year={2013},
  publisher={Springer}
}

@book{wolsey1999integer,
  title={Integer and combinatorial optimization},
  author={Wolsey, Laurence A and Nemhauser, George L},
  year={1999},
  publisher={John Wiley \& Sons}
}

@article{meyer2005convex,
  title={Convex envelopes for edge-concave functions},
  author={Meyer, Clifford A and Floudas, Christodoulos A},
  journal={Mathematical programming},
  volume={103},
  pages={207--224},
  year={2005},
  publisher={Springer}
}

@article{rikun1997convex,
  title={A convex envelope formula for multilinear functions},
  author={Rikun, Anatoliy D},
  journal={Journal of Global Optimization},
  volume={10},
  number={4},
  pages={425--437},
  year={1997},
  publisher={Springer}
}

@article{bienstock2018lp,
  title={LP formulations for polynomial optimization problems},
  author={Bienstock, Daniel and Munoz, Gonzalo},
  journal={SIAM Journal on Optimization},
  volume={28},
  number={2},
  pages={1121--1150},
  year={2018},
  publisher={SIAM}
}

@article{laurent2014handelman,
  title={Handelman’s hierarchy for the maximum stable set problem},
  author={Laurent, Monique and Sun, Zhao},
  journal={Journal of Global Optimization},
  volume={60},
  pages={393--423},
  year={2014},
  publisher={Springer}
}

@article{laurent2023effective,
  title={An effective version of Schm{\"u}dgen’s Positivstellensatz for the hypercube},
  author={Laurent, Monique and Slot, Lucas},
  journal={Optimization Letters},
  volume={17},
  number={3},
  pages={515--530},
  year={2023},
  publisher={Springer}
}

@book{deza1997geometry,
  title={Geometry of cuts and metrics},
  author={Deza, Michel and Laurent, Monique and Weismantel, Robert},
  volume={2},
  year={1997},
  publisher={Springer}
}

@article{atamturk2022submodular,
  title={Submodular function minimization and polarity},
  author={Atamt{\"u}rk, Alper and Narayanan, Vishnu},
  journal={Mathematical Programming},
  pages={1--11},
  year={2022},
  publisher={Springer}
}

@article{bach2019submodular,
  title={Submodular functions: from discrete to continuous domains},
  author={Bach, Francis},
  journal={Mathematical Programming},
  volume={175},
  pages={419--459},
  year={2019},
  publisher={Springer}
}

@article{parrilo2003semidefinite,
 Author = {Parrilo, Pablo A.},
 Title = {Semidefinite programming relaxations for semialgebraic problems},
 Journal = {Mathematical Programming},
 Volume = {96},
 Number = {2 (B)},
 Pages = {293--320},
 Year = {2003},
}

@article{dressler2022optimization,
  title={Optimization over the Boolean hypercube via sums of nonnegative circuit polynomials},
  author={Dressler, Mareike and Kurpisz, Adam and De Wolff, Timo},
Journal = {Foundations of Computational Mathematics},
 Volume = {22},
 Number = {2},
 Pages = {365--387},
 Year = {2022}
}

@article{lovasz1991cones,
  title={Cones of matrices and set-functions and 0--1 optimization},
  author={Lov{\'a}sz, L{\'a}szl{\'o} and Schrijver, Alexander},
  journal={SIAM Journal on Optimization},
  volume={1},
  number={2},
  pages={166--190},
  year={1991},
  publisher={SIAM}
}

@article{slot2023sum,
  title={Sum-of-squares hierarchies for binary polynomial optimization},
  author={Slot, Lucas and Laurent, Monique},
  journal={Mathematical Programming},
  volume={197},
  number={2},
  pages={621--660},
  year={2023},
  publisher={Springer}
}

@incollection{chlamtac2012convex,
  author = {E.~Chlamtac and M.~Tulsiani},
  title = {Convex Relaxations and Integrality Gaps},
  editor = {M.~Anjos and J.~Lasserre},
  booktitle = {Handbook on Semidefinite, Conic and Polynomial Optimization},
  series = {ISOR}, 
  volume = {166},
  pages = {139-169},
  publisher = {Springer},
  address = {New York, NY},
  year = {2012}
}

@article{gouveia2010theta,
  title={Theta bodies for polynomial ideals},
  author={Gouveia, Joao and Parrilo, Pablo A and Thomas, Rekha R},
  journal={SIAM Journal on Optimization},
  volume={20},
  number={4},
  pages={2097--2118},
  year={2010},
  publisher={SIAM}
}

@article{dressler2017positivstellensatz,
  title={A positivstellensatz for sums of nonnegative circuit polynomials},
  author={Dressler, Mareike and Iliman, Sadik and De Wolff, Timo},
  journal={SIAM Journal on Applied Algebra and Geometry},
  volume={1},
  number={1},
  pages={536--555},
  year={2017},
  publisher={SIAM}
}

@article{ghasemi2012lower,
  title={Lower bounds for polynomials using geometric programming},
  author={Ghasemi, Mehdi and Marshall, Murray},
  journal={SIAM Journal on Optimization},
  volume={22},
  number={2},
  pages={460--473},
  year={2012},
  publisher={SIAM}
}

@book{conforti2014integer,
author="Conforti, Michele
and Cornu{\'e}jols, G{\'e}rard
and Zambelli, Giacomo",
 Title = {Integer Programming},
 Series = {Graduate Texts in Mathematics},
 Volume = {271},
 Year = {2014},
 Publisher = {Cham: Springer},
 Language = {English}
}

@article{billionnet1985maximizing,
  title={Maximizing a supermodular pseudoboolean function: A polynomial algorithm for supermodular cubic functions},
  author={Billionnet, Alain and Minoux, Michel},
  journal={Discrete Applied Mathematics},
  volume={12},
  number={1},
  pages={1--11},
  year={1985},
  publisher={Elsevier}
}

@article{orlin2009faster,
  title={A faster strongly polynomial time algorithm for submodular function minimization},
  author={Orlin, James B},
  journal={Mathematical Programming},
  volume={118},
  number={2},
  pages={237--251},
  year={2009},
  publisher={Springer}
}

@article{putinar1993positive,
  title={Positive polynomials on compact semi-algebraic sets},
  author={Putinar, Mihai},
  journal={Indiana University Mathematics
Journal},
  volume={42},
  number={3},
  pages={969--984},
  year={1993},
  publisher={JSTOR}
}

@article{goemans1995improved,
  title={Improved approximation algorithms for maximum cut and satisfiability problems using semidefinite programming},
  author={Goemans, Michel X and Williamson, David P},
  journal={Journal of the ACM},
  volume={42},
  number={6},
  pages={1115--1145},
  year={1995},
  publisher={ACM New York, NY, USA}
}

@article{laurent2009sums,
  title={Sums of squares, moment matrices and optimization over polynomials},
  author={Laurent, Monique},
  journal={Emerging applications of algebraic geometry},
  pages={157--270},
  year={2009},
  publisher={Springer}
}

@article{sakaue2017exact,
  title={Exact semidefinite programming relaxations with truncated moment matrix for binary polynomial optimization problems},
  author={Sakaue, Shinsaku and Takeda, Akiko and Kim, Sunyoung and Ito, Naoki},
  journal={SIAM Journal on Optimization},
  volume={27},
  number={1},
  pages={565--582},
  year={2017},
  publisher={SIAM}
}

@article{xu2023,
  title={Submodular maximization and its generalization through an intersection cut lens},
  author={Xu, Liding and Liberti, Leo},
  journal={Mathematical Programming},
  volume = {},
  number = {},
  pages={},
  doi = {10.1007/s10107-024-02059-2},
  year={2024},
  publisher={Springer}
}

@inproceedings{lee2015lower,
author = {Lee, James R. and Raghavendra, Prasad and Steurer, David},
title = {Lower Bounds on the Size of Semidefinite Programming Relaxations},
year = {2015},
publisher = {ACM},
address = {New York, NY},
pages = {567–576},
numpages = {10},
booktitle = {Proceedings of STOC 2015}
}

@techreport{nguyen2013deriving,
  title={Deriving the convex hull of a polynomial partitioning set through lifting and projection},
  author={Nguyen, Trang T and Richard, Jean-Philippe P and Tawarmalani, Mohit},
  year={2013},
  institution={Technical report, working paper}
}

@article{gallo1989supermodular,
  title={On the supermodular knapsack problem},
  author={Gallo, Giorgio and Simeone, Bruno},
  journal={Mathematical Programming},
  Number = {2 (B)},
  volume={45},
  pages={295--309},
  year={1989},
  publisher={Springer}
}

@phdthesis{hansen1974lin,
  title={Programmes math{\'e}matiques en variables 0-1},
  author={Hansen, Pierre},
  year={1974},
  school={Universit{\'e} libre de Bruxelles, Facult{\'e} des sciences appliqu{\'e}es}
}

@inproceedings{orlin2013max,
author = {Orlin, James B.},
title = {{Max flows in O(nm) time, or better}},
year = {2013},
publisher = {ACM},
address = {New York, NY},
booktitle = {Proceedings of STOC 2013},
pages = {765–774},
numpages = {10},
location = {Palo Alto, California, USA}
}

@article{glover1973further,
  title={Further reduction of zero-one polynomial programming problems to zero-one linear programming problems},
  author={Glover, Fred and Woolsey, Eugene},
  journal={Operations Research},
  volume={21},
  number={1},
  pages={156--161},
  year={1973},
  publisher={INFORMS}
}

@article{lasserre2002explicit,
  title={An explicit equivalent positive semidefinite program for nonlinear 0-1 programs},
  author={Lasserre, Jean B},
  journal={SIAM Journal on Optimization},
  volume={12},
  number={3},
  pages={756--769},
  year={2002},
  publisher={SIAM}
}

@article{lasserre2002polynomials,
  title={Polynomials nonnegative on a grid and discrete optimization},
  author={Lasserre, Jean B},
  journal={Transactions of the American Mathematical Society},
  volume={354},
  number={2},
  pages={631--649},
  year={2002}
}

@article{bienstock2004subset,
  title={Subset algebra lift operators for 0-1 integer programming},
  author={Bienstock, Daniel and Zuckerberg, Mark},
  journal={SIAM Journal on Optimization},
  volume={15},
  number={1},
  pages={63--95},
  year={2004},
  publisher={SIAM}
}

@article{balas1993lift,
  title={A lift-and-project cutting plane algorithm for mixed 0--1 programs},
  author={Balas, Egon and Ceria, Sebasti{\'a}n and Cornu{\'e}jols, G{\'e}rard},
  journal={Mathematical Programming},
  volume={58},
  Number = {3 (A)},
  pages={295--324},
  year={1993},
  publisher={Springer}
}

@article{handelman1988representing,
  title={Representing polynomials by positive linear functions on compact convex polyhedra},
  author={Handelman, David},
  journal={Pacific Journal of Mathematics},
  volume={132},
  number={1},
  pages={35--62},
  year={1988},
  publisher={Mathematical Sciences Publishers}
}

@article{boros2002pseudo,
  title={Pseudo-boolean optimization},
  author={Boros, Endre and Hammer, Peter L},
  journal={Discrete Applied mathematics},
  volume={123},
  number={1-3},
  pages={155--225},
  year={2002},
  publisher={Elsevier}
}

@book{de2012algebraic,
 Author = {De Loera, Jes{\'u}s A. and Hemmecke, Raymond and K{\"o}ppe, Matthias},
 Title = {Algebraic and geometric ideas in the theory of discrete optimization},
 Volume = {14},
 Year = {2013},
Series = {MOS/SIAM Series on Optimization},
 Publisher = {Philadelphia, PA: SIAM}
}

@book{lasserre2015introduction,
 Author = {Lasserre, Jean Bernard},
 Title = {An introduction to polynomial and semi-algebraic optimization},
 Series = {Cambridge Texts in Applied Mathematics},
 Year = {2015},
 Publisher = {Cambridge: Cambridge University Press},
 Language = {English}
}

@book{parrilo2020sum,
 Editor = {Parrilo, Pablo A. and Thomas, Rekha R.},
 Title = {Sum of squares: theory and applications. {AMS} short course, {Baltimore}, {MD}, {USA}, {January} 14--15, 2019},
 Series = {Proceedings of Symposia in Applied Mathematics},
 Volume = {77},
 Year = {2020},
 Publisher = {Providence, RI: AMS},
}

@book{crama2011boolean,
 author = {Crama, Yves and Hammer, Peter L.},
 Title = {Boolean functions. {Theory}, algorithms, and applications},
 Volume = {142},
 Year = {2011},
 Publisher = {Cambridge: Cambridge University Press},
Series = {Encyclopedia of Mathematics and Its Applications}
}

@book{magron2023sparse,
 Author = {Magron, Victor and Wang, Jie},
 Title = {Sparse polynomial optimization. {Theory} and practice},
 Series = {Series on Optimization and its Applications},
 Volume = {5},
 Year = {2023},
 Publisher = {Singapore: World Scientific},
}

@article{waki2006sums,
  title={Sums of squares and semidefinite program relaxations for polynomial optimization problems with structured sparsity},
  author={Waki, Hayato and Kim, Sunyoung and Kojima, Masakazu and Muramatsu, Masakazu},
  journal={SIAM Journal on Optimization},
  volume={17},
  number={1},
  pages={218--242},
  year={2006},
  publisher={SIAM}
}

@article{lasserre2006convergent,
  title={Convergent {SDP}-relaxations in polynomial optimization with sparsity},
  author={Lasserre, Jean B},
  journal={SIAM Journal on Optimization},
  volume={17},
  number={3},
  pages={822--843},
  year={2006},
  publisher={SIAM}
}

@article{grimm2007note,
  title={A note on the representation of positive polynomials with structured sparsity},
  author={Grimm, David and Netzer, Tim and Schweighofer, Markus},
  journal={Archiv der Mathematik},
  volume={89},
  number={5},
  pages={399--403},
  year={2007},
  publisher={Springer}
}

@book{murota1998discrete,
 Author = {Murota, Kazuo},
 Title = {Discrete convex analysis},
 Series = {SIAM Monographs on Discrete Mathematics and Applications},
 Volume = {10},
 Year = {2003},
 Publisher = {Philadelphia, PA: SIAM}
}

@article{crama1989recognition,
  title={Recognition problems for special classes of polynomials in 0--1 variables},
  author={Crama, Yves},
  journal={Mathematical Programming},
  volume={44},
  pages={139--155},
  year={1989},
  publisher={Springer}
}

@article{picard1982network,
  title={A network flow solution to some nonlinear 0-1 programming problems, with applications to graph theory},
  author={Picard, Jean-Claude and Queyranne, Maurice},
  journal={Networks},
  volume={12},
  number={2},
  pages={141--159},
  year={1982},
  publisher={Wiley Online Library}
}

@incollection{lovasz1983submodular,
author="Lov{\'a}sz, L.",
editor="Bachem, Achim
and Korte, Bernhard
and Gr{\"o}tschel, Martin",
title="Submodular functions and convexity",
bookTitle="Mathematical Programming The State of the Art: Bonn 1982",
year="1983",
publisher="Springer Berlin Heidelberg",
address="Berlin, Heidelberg",
pages="235--257"
}

@article{krivine1964anneaux,
  title={Anneaux pr{\'e}ordonn{\'e}s},
 Author = {Krivine, J. L.},
 Journal = {Journal d’Analyse Mathm{\'e}atique},
 Volume = {12},
 Pages = {307--326},
 Year = {1964}
}


\end{document}